\documentclass[11pt]{article}
\usepackage{xcolor}
\usepackage{amsmath}
\usepackage{hyperref}
\usepackage{amssymb}
\usepackage{bbm}
\usepackage{enumerate}   
\usepackage{geometry}
\usepackage{amsfonts}
\usepackage{amsthm} 

\newtheorem{theorem}{Theorem}
\newtheorem{lemma}[theorem]{Lemma}

\newtheorem{proposition}[theorem]{Proposition}
\newtheorem{mydef}[theorem]{Definition}

\newtheorem{remark}[theorem]{Remark}
\newtheorem{assumption}[theorem]{Assumption}

\newcommand{\E}{\mathbb{E}}
\newcommand{\R}{\mathbb{R}}
\newcommand{\PP}{\mathbb{P}}
\newcommand{\Q}{\mathbb{Q}}
\newcommand{\N}{\mathbb{N}}
\newcommand{\1}{\mathbbm{1}}
\newcommand{\tX}{\tilde X}
\newcommand{\tB}{\tilde B}
\newcommand{\dd}{{\rm d}}
\newcommand{\Aa}{\mathcal A}
\newcommand{\Ss}{\mathcal S}
\newcommand{\Kk}{\mathcal K}
\newcommand{\Ff}{\mathcal F}
\newcommand{\Pp}{\mathcal P}
\newcommand{\Uu}{\mathcal U}
\newcommand{\e}{\varepsilon}
\newcommand{\sm}{{s-}}
\newcommand{\vip}{\vskip5pt}

\title{Asymptotics of a two-species particle system associated to the doubly parabolic Keller-Segel equation in the plane}

\author{Nicolas Fournier\footnote{Sorbonne Université and Université Paris Cité, CNRS, 
Laboratoire de Probabilités, Statistique et Modélisation, F-75005 Paris, France.}
\hskip2mm  and Milica Toma\v sevi\'c\footnote{CMAP, CNRS, École polytechnique, Institut Polytechnique de
Paris, 91120 Palaiseau, France.}}

\date{}
\begin{document}

\maketitle

\begin{abstract}
 We consider the two-species particle system introduced by Stevens~\cite{Stevens} related to the doubly parabolic Keller-Segel equation.  It consists of $N$ cells and of a varying number of chemoattractant particles. Cells  diffuse in the plane and follow the (mollified) empirical gradient of concentration of chemoattractant. Chemoattractant particles are produced by  cells at some constant rate, diffuse and  disappear at some constant rate. We show that when the sensitivity of cells to the chemoattractant is small enough, under some rather weak condition on the family of mollifiers, this system  approximates the parabolic-parabolic Keller-Segel equation as $N\to \infty$. We also prove that when $N$ is fixed and when the production rate of chemoattractant particles tends to infinity, this system approximates the (non-Markovian) one-species system introduced in~\cite{Mi-De} and further studied in~\cite{FT1}.
\end{abstract}

\vip


\section{Introduction}

\paragraph{Motivation.}
Chemotaxis is a fundamental biological process that allows cells and organisms to navigate their environment by moving toward favorable conditions and away from harmful ones. Found across a wide range of life forms, from bacteria to human cells, it involves two key steps. First, a chemical spreads {\it via} diffusion, creating a concentration gradient. Second, each cell or organism detects this gradient and adjusts its movement accordingly, either advancing toward higher concentrations (in the case of a chemoattractant) or retreating from them (in the case of a chemorepellent).
This biological mechanism can be formalized mathematically through the Keller–Segel model, which provides a \textit{macroscopic} description of chemotaxis. Rather than tracking individual cells, the model assumes a large-population regime in which the number of cells is sufficiently high to justify a continuum description in terms of a cell density. It captures how this density evolves under the combined effects of  diffusion and directed motion along chemical gradients. Within this framework, cells both diffuse and attract each other in response to the chemical signals they sense or produce. For the original motivation and detailed biological background, we refer to the foundational works of Keller and Segel~\cite{KELLER_SEGEL1, KELLER_SEGEL2, KELLER_SEGEL3}.
 The aim of the present article is to revisit the corresponding \textit{microscopic} description and to bridge the gap between two distinct approaches that have emerged in the probabilistic literature. One is based on a two-species particle system introduced by Stevens~\cite{Stevens}, while the other relies on a single-species non-Markovian system proposed more recently by Talay-Tomašević~\cite{Mi-De} and further analyzed in~\cite{JTT,tomasevic.these,FT1}.

\paragraph{The Keller-Segel equation.} The so-called \textit{minimal model} couples the evolution of the cell density $(\rho_t(x))_{t\geq 0, x \in \R^2}$ and chemoattractant concentration $(c_t(x))_{t\geq 0, x \in \R^2}$ in the following way
\begin{equation}
\label{EDP_KS}
\begin{cases}
&\partial_t \rho_t(x) = \Delta  \rho_t(x)-\chi \nabla \cdot (\rho_t(x)\nabla c_t(x)), 
\hskip1.5cm  t>0,\quad x\in\R^2, \\[4pt]
&\theta \partial_t c_t(x) = \Delta c_t(x)-\lambda c_t(x)+\rho_t(x), \hskip2.15cm t>0, \quad x\in\R^2,
\end{cases}
\end{equation}
with $\rho_0$ and $c_0$ given.  We will assume that $\rho_0$ is a probability measure on $\R^2$ and that
$c_0\in L^1(\R^2)\cap L^\infty(\R^2)$ is nonnegative.
The parameters $\chi>0$, $\theta\geq 0$ and 
$\lambda\geq 0$ respectively stand for the sensitivity of cells to the chemoattractant, the ratio between the diffusion time scales of bacteria and chemoattractant, and the decay rate of the chemoattractant. For a thorough review on this model, see the paper of Horstmann~\cite{Horstmann1}.
A notable feature of the Keller–Segel model is its ability to produce singular behavior in finite time. This phenomenon reflects the tendency of cells to self-organize into sharply concentrated clusters under the influence of chemotactic attraction. In particular, for sufficiently regular initial data, the parabolic–elliptic case (corresponding to $\theta = 0$) displays a critical threshold: At least for $\lambda=0$, when $\chi$ exceeds $8\pi$, solutions undergo finite-time blow-up, whereas for $\chi \leq 8\pi$, solutions remain globally well-defined. This was first rigorously established by Nagai~\cite{nagai} ($\chi<8\pi$, radial case),
Blanchet-Dolbeault-Perthame~\cite{bdp} ($\chi<8\pi$) and Biler-Karch-Lauren\c{c}ot-Nadzieja~\cite{bkln} 
($\chi\leq 8\pi$, radial case). In the parabolic-parabolic case where $\theta>0$, the  global well-posedness still holds for $\chi<8\pi$ and any reasonable initial data, see Calvez-Corrias~\cite{CalCor2008}. However, in the special case of $c_0= 0$ and for any $\chi>0$,
the global well-posedness holds true when $\theta>0$ is large enough, see  Biler-Guerra-Karch~\cite{Biler-Guerra-Karch}, and this was extended to a more general class of initial  concentrations $c_0$ (with a smallness condition depending on $\theta$) by  
Corrias-Escobedo-Matos~\cite{Corrias2014}. 
Concerning explosion, the situation is still largely open, let us mention that radial solutions 
on a disk in $\R^2$ blowing-up for $\chi>8\pi$  have been  exhibited by  Herrero-Velasquez~\cite{HerreroandVelazquez}.  In addition, a criterion for the explosion of radial solutions has been obtained by Mizoguchi~\cite{mizoguchiBlowUp}: The conditions are that $\chi>8\pi$ and  an energy condition on the initial data $(\rho_0,c_0)$. 

\paragraph{Particle approximation in the parabolic-elliptic case.}
In the elliptic case where $\theta=0$, one naturally approximates the Keller-Segel equation by a finite system of Brownian motions interacting through a pairwise attraction in $1/r$. Such a particle system may explode due to the occurrence of many collisions, a phenomenon studied in~\cite{FournierTardy}. It is thus rather challenging to justify this approximation, but this is now rather well-understood in the subcritical case $\chi\leq 8\pi$, see \cite{fournier-jourdain}, Bresch-Jabin-Wang~\cite{BJW}  and Tardy~\cite{T}. However, a true convergence result (not along subsequences) in the whole plane (not in the torus) is still unavailable. See also the work~\cite{ORT} on moderately interacting particle systems. 

\paragraph{A two-species particle approximation in the doubly parabolic case.}
In the seminal paper~\cite{Stevens}, Stevens introduces a two species particle approximation of~\eqref{EDP_KS} for $\theta>0$. Let us describe informally the variant of this system we will study in the present paper, see Section~\ref{secmr} for a rigorous description. Fix an integer $N\geq 2$ and a real $M>0$. We have $N$ cells and a temporally varying number of chemoattractant particles. 

\vip
\noindent $\bullet$ Initially, the $N$ cell positions are i.i.d. and $\rho_0$-distributed, and we have around 
$NM||c_0||_{L^1(\R^2)}$ chemoattractant particles with i.i.d. $||c_0||_{L^1(\R^2)}^{-1}c_0$-distributed positions.
\vip
\noindent $\bullet$ Each chemoattractant particle disappears at instantaneous rate $\theta^{-1}\lambda$.
\vip
\noindent $\bullet$ Each cell produces chemoattractant particles according to a Poisson process with rate $\theta^{-1}M$. When a chemoattractant particle is born, it takes the position of its mother cell.
\vip
\noindent $\bullet$ Each chemoattractant particle moves like a Brownian motion with diffusion coefficient $\sqrt{2 \theta^{-1}}$.
\vip
\noindent $\bullet$ Each cell moves like a Brownian motion with diffusion coefficient $\sqrt 2$ which drifts according to the concentration gradient of chemoattractant particles. Since such particles are finitely many, a mollification $\phi_\e$ is necessary. Namely, the instantaneous drift of the $i$-th cell at time $t$ is given by
$\nabla \phi_\e \ast \nu^{i}_t$ evaluated at the position of the $i$-th cell and $\nu^i_t$ is the empirical concentration at time $t$ of chemoattractant particles, excluding those produced by the $i$-th cell and
normalized by $M(N-1)$.

\vip

We denote by $(X^{i,N,M,\e}_t)_{t\geq 0, i \in I_N}$ the process 
describing the cell positions, where $I_N=\{1,\dots,N\}$.
See Assumption~\ref{mol} for precise conditions on the family of mollifiers $(\phi_\e, \e\in (0,1])$, we mainly assume that $\phi_\e(x)=\e^{-1}\phi(\e^{-\frac12}x)$, for some smooth probability density $\phi$ on $\R^2$.

\paragraph{The result of~\cite{Stevens}.}
 Although Stevens works in any dimension and with a more general chemotaxis PDE, let us specify her results in our context. She considers (a variant of) the particle system described in the previous paragraph and assumes (i) that the production rate of chemoattractant particles 
$M$ does not depend on the number $N$ of cells, 
(ii) that \eqref{EDP_KS} admits a unique solution such that both $(\rho_t)_{t\in[0,T]}$ and $(c_t)_{t\in[0,T]}$ belong to $C_b^{1,3}([0,T]\times \R^2)\cap C^0([0,T], L^2(\R^2))$, 
(iii) a set of conditions on the mollification parameter implying that, with our notation, 
$\lim_N N \e_N^{15+}=\infty$ (see~\cite[Sections~2 and~3]{Stevens}, she takes $\e_N=\hat \alpha_N^{-2}$ 
with $\hat \alpha_N=N^{\hat \alpha/2}$ with $\hat\alpha<\frac1{15}$).
Under these conditions, the mollified empirical measures of cells and of chemoattractant particles converge in probability to $(\rho_t)_{t\in [0,T]}$ and $(c_t)_{t\in [0,T]}$, as $N\to \infty$.
\vip
Condition (i) is not very important, it is likely that the case where $M$ varies with $N$ could be treated by the same method. Although we found no precise reference, condition (ii) seems rather reasonable and allows one to treat any values of $\chi>0$ and $\theta>0$ for which  there exists a unique
%(possibly local in time) 
smooth solution to \eqref{EDP_KS}, up to the maximal existence time of such solution.
\vip
Our goal with respect to~\cite{Stevens} is mainly to weaken the condition (iii) on the mollification parameter. Roughly, we will be able (e.g. in the case $M=1$) to include the case where $\lim_N N \e_N^{3+}=\infty$. 
However, our convergence result is weaker in that it holds only along subsequences and we can only work with global solutions, hence we need to assume that $\chi>0$ is small enough.

\paragraph{A one-species particle approximation in the doubly parabolic case.}
Here we present the alternative microscopic approximation by Talay-Tomašević~\cite{Mi-De}.
See Section~\ref{secmr} and~\cite[Section~1]{FT1} for details.
It does not involve any smoothing, but it only tracks the cells positions. Using Duhamel's formula for the chemoattractant concentration $(c_t)_{t\geq 0}$ with $(\rho_t)_{t\geq 0}$ as a source term, one can interpret the equation for $(\rho_t)_{t\geq 0}$ as a Fokker-Plank equation of the following non-linear process:
\begin{equation}\label{eq:NLSDE}
X_t= X_0 + \sqrt{2} B_t + \chi \int_0^t \nabla b_s^{c_0,\theta,\lambda}(X_s)\dd s 
+ \chi \int_0^t\int_0^s (\nabla K^{\theta,\lambda}_{s-u}\ast \rho_u)(X_s) \dd u\dd s,
\quad \rho_s={\rm Law}(X_s),
\end{equation}
where, for $ (t,x) \in (0,\infty)\times \R^2$,
\begin{equation}\label{gkb}
g_t(x):=\frac{1}{2\pi t} e^{-\frac{|x|^2}{2t}}, \quad 
K_t^{\theta,\lambda}(x):= \frac{1}{\theta} e^{-\frac{\lambda}{\theta}t}g_{2t/\theta}(x),  
\quad b_t^{c_0,\theta,\lambda} (x):= \theta (K^{\theta,\lambda}_t \ast  c_0)(x).
\end{equation}
The process \eqref{eq:NLSDE} represents the position of a  typical cell in a macroscopic population (of infinite size). A natural discretization of~\eqref{eq:NLSDE} is the following particle system: setting $I_N=\{1,\dots, N\}$,
\begin{align}\label{PS}
X_t^{i,N}=& X_0^{i,N} + \sqrt{2} B_t^i + \chi \int_0^t \nabla b_s^{c_0,\theta,\lambda}(s,X_s^{i,N})\dd s \\
&+ \frac{\chi}{N-1} \sum_{j \in I_N,j\neq i}
\int_0^t \int_0^s \nabla K^{\theta,\lambda}_{s-u} ( X^{i,N}_s-X^{j,N}_u)  \dd u \dd s,
\quad i \in I_N.\notag
\end{align}
The rigorous derivation of the Keller-Segel system \eqref{EDP_KS} from such a particle system is quite challenging due to the nature of interaction between the particles which is both non-Markovian and singularly attractive. Biologically speaking, one may understand the memory in the interaction as follows: at some current time $s$, cells feel the chemoattractant produced at all past instants $u<s$ by all of the other cells, which has evolved (through the heat kernel corrected with the decay rate) between $u$ and $s$.
It also feels the initial chemoattractant particles through $b_s^{c_0,\theta,\lambda}$.
\vip
In dimension $1$, the propagation  of chaos of \eqref{PS} towards \eqref{eq:NLSDE} (without any restriction on the parameters) is obtained  in \cite{JTT}  by means of Girsanov transformation.

\paragraph{The main result of~\cite{FT1}.}
In dimension $2$, we have shown that for each $\theta>0$,  there exists some $\chi_{\theta}^*>0$ such that for all $\lambda>0$, all $\rho_0\in \Pp(\R^2)$, all nonnegative $c_0\in L^\infty(\R^2)$, all $\chi \in (0,\chi^*_\theta)$, there exists an exchangeable $N$-Keller-Segel particle system~\eqref{PS} for each $N\geq 2$ with i.i.d. $\rho_0$-distributed initial cell positions. Moreover, the family of empirical measures 
$$
\Big(\mu^N=\frac1N\sum_{i=1}^N \delta_{(X^{i,N}_t)_{t\geq 0}}, N \geq 2\Big)
$$  
is tight in $\Pp(C(\R_+,\R^2))$ and any (possibly random) limit point
of $(\mu^N)_{N\geq 2}$ a.s. solves the martingale problem related to~\eqref{eq:NLSDE}.

\vip
The threshold $\chi^*_\theta$ is recalled in Remark~\ref{chistar} 
below and we have (see~\cite[Remark~15]{FT1} with $p=\infty$) 
\begin{equation}\label{estth}
\liminf_{\theta\to 0} \chi_{\theta}^* \geq 3.28, \quad \chi_{1}^*\geq 1.39 \quad
\text{and} \quad \liminf_{\theta\to \infty} \sqrt\theta\chi_{\theta}^* \geq 1.65. 
\end{equation}
This is clearly not optimal, since we expect the result to hold for all $\chi\in (0,8\pi)$ (for all $\theta>0$) and even for larger values of $\chi$ if $c_0=0$ and if $\theta$ is large. 
\vip
We actually suppose a weaker condition than $c_0 \in L^\infty(\R^2)$ in~\cite{FT1} and some more general initial conditions for~\eqref{PS}, but this is not very important for the present work.
\vip
With respect to~\cite{FT1}, our goal is to obtain the one-species system as a $M\to \infty$ limit of the two-species system.

\paragraph{First main result (informal).}
We first study the link between the two-species and the one-species
particle systems. Assume that $\rho_0 \in \Pp(\R^2)$ and $c_0 \in L^1(\R^d)\cap L^\infty(\R^2)$,
that $\theta>0$, $\lambda>0$ and $\chi \in (0,\chi^*_\theta)$. Fix $N\geq 2$ and assume that
the $(0,1]$-valued sequence $(\e_n)_{n\geq 1}$ and the $(0,\infty)$-valued sequence $(M_n)_{n\geq 1}$ satisfy
\begin{equation}\label{ccc1}
\text{for some $\sigma \in \Big(0,\frac12\Big]$},\quad (a) \; \lim_n M_n=\infty, \quad (b)\; \lim_n\e_n=0,\quad (c)\; \lim_n M_n\e_n^{3+\sigma}=\infty.
\end{equation}
Then, the family 
$((X^{i,N,M_n,\e_n}_t)_{t\geq 0, i \in I_N}, n\geq 1)$ is tight and any limit point solves~\eqref{PS}.

\vip
Notice that \eqref{ccc1} holds if e.g. $\lim_n M_n=\infty$ and $\e_n=M_n^{-\alpha}$ for some 
$\alpha\in (0,\frac 13)$.

\paragraph{Second main result (informal).}
We next study the convergence of the two-species system to the Keller-Segel equation.
Assume that $\rho_0 \in \Pp(\R^2)$ and $c_0 \in L^1(\R^d)\cap L^\infty(\R^2)$,
that $\theta>0$, $\lambda>0$ and $\chi \in (0,\chi^*_\theta)$. Assume that
the $(0,1]$-valued sequence $(\e_N)_{N\geq 2}$ and the $(0,\infty)$-valued sequence $(M_N)_{N\geq 2}$ satisfy
\begin{equation}\label{ccc2}
\text{for some $\sigma\in \Big(0,\frac12\Big]$},\quad
(a)\; \lim_N \e_N=0, \quad (b)\; \lim_N NM_N\e_N^{3+\sigma}=\infty, \quad
(c) \;\lim_N NM_N^{1+\sigma}\e_N^{2+\sigma}=\infty.
\end{equation}
Then the family of empirical measures 
$$
\Big(\mu^N:=\frac1N\sum_{i=1}^N \delta_{(X^{i,N,M_N,\e_N}_t)_{t\geq 0}} , N\geq 2\Big)
$$ 
is tight in 
$\Pp(C(\R_+,\R^2))$ and any (possibly random) limit point
of $(\mu^N)_{N\geq 2}$ a.s. solves the martingale problem related to~\eqref{eq:NLSDE}.

\vip
We emphasize that we do not need $M_N$ to be bounded from below: Our result requires that $NM_N$, which represents the order of magnitude of the (varying) number of chemoattractant particles, tends to infinity. If $M_N$ is bounded from below, then~\eqref{ccc2}-(c) is implied by~\eqref{ccc2}-(a)-(b). 
Although we can find, using logarithms, some cases where~\eqref{ccc2}-(a)-(b) hold while~\eqref{ccc2}-(c) does not, the following example shows
that~\eqref{ccc2}-(c) is rather harmless.
\vip

If $M_N=N^\beta$ and $\e_N=N^{-\alpha}$ for some $\beta \in (-1,\infty)$ and some $\alpha>0$, then~\eqref{ccc2}-(b) is satisfied if and only if $\alpha \in (0,\frac{1+\beta}3)$ (choose $\sigma=(\frac{1+\beta}{2\alpha}-\frac 32)\land\frac12>0$), which implies that~\eqref{ccc2} is entirely satisfied
(choose $\sigma=(\frac{1+\beta}{2\alpha}-\frac 32)\land(\frac{1+\beta-2\alpha}{\alpha-\beta})\land\frac12>0$ if $\beta<0$
and $\sigma=(\frac{1+\beta}{2\alpha}-\frac 32)\land\frac12>0$ if $\beta\geq 0$).

\paragraph{Discussion.}
Let us first mention that we find the same threshold $\chi_\theta^*$ as in~\cite{FT1} when we studied the one-species system. As already mentioned, this threshold is clearly not optimal, but, at least, the increased complexity of the two-species system does not lead to a deteriorated threshold.

\vip 
Let us now discuss the mollification parameter. For $(Z_i)_{i\geq 1}$ i.i.d., valued in $\R^2$, with a smooth fast decaying density $f$ and for $\xi_\ell=\frac1\ell\sum_{i=1}^\ell \delta_{Z_i}$, it holds that $(\xi_\ell \ast \nabla \phi_{\e_\ell})(x) \to \nabla f(x)$ in probability as $\ell\to \infty$ for all $x\in \R^2$ as soon as $\e_\ell\to 0$ 
and $\lim_{\ell} \ell \e_\ell^{2}=\infty$, and these conditions are necessary. This can be easily seen through a bias/variance study, see Stone~\cite{stone}.
\vip

In our framework, we need that in some sense, $(\nu_t\ast \nabla \phi_\e) (x)$ approximates $\nabla c_t(x)$,
where $\nu_t$ is the empirical concentration of chemoattractant particles. The order of magnitude of the (varying) number of those particles is $NM$. Although this is not clear, it might be possible to replace conditions (b)-(c) in~\eqref{ccc1} by $\lim_n M_n \e_n^2=\infty$ and conditions (b)-(c) in~\eqref{ccc2} by  $\lim_N NM_N \e_N^2=\infty$.
We do not see at all how to reach such a weak condition, mainly due to 
Step~3 of the proof of Proposition~\ref{nest}. Leaving aside complex computations, we notice that chemoattractant particles are far from being i.i.d. Brownian motions as they are correlated through their common ancestor and their initial positions. This may explain why one could not reach the best possible condition.

\vip

With more work, it is likely that we could include, as in~\cite{FT1}: The case $\lambda=0$ (which is actually slightly simpler), the case where we only have $c_0 \in L^p(\R^2)$ for some $p>2$ (with a worse threshold $\chi^*_{\theta,p}$), the case where the initial cell and chemoattractant positions are not i.i.d. but enjoy some some exchangeability/consistency.

\paragraph{Plan of the paper.}
In Section~\ref{secmr}, we give some precise definitions of the objects presented in this introduction as well as rigorous statements of our main results, see Theorem~\ref{main}.
\vip
In Section~\ref{spre}, we recall a crucial functional inequality obtained in~\cite{FT1} as well as some constants introduced in~\cite{FT1} to define the threshold $\chi_\theta^*$. We also check an easy Itô formula concerning chemoattractant particles.
\vip
In the two-species system, cells interact indirectly, through chemoattractant particles that they produce. Hence Section~\ref{fcat} is important: we quantify how far is the drift of the cell particles from a trajectory-dependent drift (as in the one-species system). This is where the conditions on the mollification parameter arise.
\vip
In Section~\ref{meat}, we obtain a uniform in $N,M,\e$ control  of the drift of cells in the two-species system
and deduce Theorem~\ref{main}-(i), i.e. tightness. In the whole section except Lemma~\ref{nest}, we adapt the computations elaborated in~\cite{FT1}. Due to the complexity of the present model, these computations are more involved.
\vip
Theorem~\ref{main}-(ii), i.e. convergence of the two-species system to the one-species system is proved in Section~\ref{tstos}.
\vip
The proof of Theorem~\ref{main}-(iii), i.e. convergence of the  two-species  system to the Keller-Segel model, is shown in Section~\ref{tstks}.
\vip
Finally, Appendix~\ref{apa} is devoted to some lemmas about mollifiers and the Gaussian kernel.

\paragraph{Acknowledgments.} The authors thank Maxence Baccara for his master internship around a regularised version of the present problem. Our discussions and his preliminary work helped us better understand the challenges of the present work.

\section{Main results}\label{secmr}

\paragraph{The nonlinear martingale problem.}
We introduce the set $\Pp(C(\R_+,\R^2))$ of probability measures on $C(\R_+,\R^2)$.
As is well-known, one may deal with the nonlinear S.D.E.~\eqref{eq:NLSDE} through the corresponding nonlinear martingale problem. Recall that $K^{\theta,\lambda}_{s}(x)$ and $b_s^{c_0,\theta,\lambda}$ were introduced in~\eqref{gkb}.

\begin{mydef}
\label{defMP}
Fix some $\rho_0 \in \Pp(\R^2)$ and some nonnegative $c_0\in L^\infty(\R^2)$. 
Consider the canonical space $C(\R_+,\R^2)$ equipped with its
canonical process $(w_t)_{t\geq 0}$ and its canonical filtration. Let $\Q \in \Pp(C(\R_+,\R^2))$
and denote $(\Q_t)_{t\geq 0}$ its family of one-dimensional time marginals. 
We say that $\Q$ solves the non-linear 
martingale problem $\mathcal{(MP)}(\rho_0,c_0)$ if
\begin{equation}
\label{eq:ConditionMP}
\int_0^t \int_{\R^2} \int_0^s \int_{\R^2}   (K^{\theta,\lambda}_{s-u}(x-y) + 
|\nabla K^{\theta,\lambda}_{s-u}(x-y)|) \Q_u(\dd y)  \dd u   \Q_s(\dd x) \dd s < \infty,
\end{equation}
and if for any $\varphi \in C_c^2(\R^2)$, the process
\begin{equation}
\label{def_mart}
M^\varphi_t:=\varphi(w_t)-f(w_0)-\int_0^t \Big[ \Delta \varphi(w_u)+\chi \nabla \varphi(w_u)\cdot 
\Big(\nabla b_s^{c_0,\theta,\lambda}(w_s) + \int_0^u (\nabla K^{\theta,\lambda}_{u-r}\ast \Q_r)(w_u) \dd r \Big) \Big]\dd u
\end{equation}
is a $\Q$-martingale.
\end{mydef}

As shown in~\cite[Section~1]{FT}, \eqref{eq:ConditionMP} implies that everything makes sense 
in~\eqref{def_mart}; and for $\Q$ a solution to $\mathcal{(MP)}(\rho_0,c_0)$,
setting $c_t=b_t^{c_0,\theta,\lambda}+ \int_0^t (K^{\theta,\lambda}_{t-s} \ast \Q_s )\,\dd s$, it holds that 
$(\Q_t,c_t)_{t\geq 0}$ is a weak solution to~\eqref{EDP_KS} in the sense of~\cite[Definition~1]{FT}.

\paragraph{A single-species particle system.}

\begin{mydef}\label{def:PS} Fix $N\geq 2$, $\rho_0 \in \Pp(\R^2)$ and some nonnegative
$c_0\in L^\infty(\R^2)$. Consider some i.i.d. family $(B_t^i)_{t\geq 0,i\in I_N}$ of $2D$-Brownian 
motion, as well as i.i.d. family 
$(X_0^i)_{i\in I_N}$ of $\R^2$-valued $\rho_0$-distributed random variables, 
independent of the family of Brownian motions. 
A family of continuous $\R^2$-valued processes $(X^{i,N}_t)_{t\geq 0,i\in I_N}$ is said to be a $(\rho_0,c_0,N)$-1SKS 
(Single Species Keller-Segel)
particle system if a.s., for all $t\geq 0$, all $i,j \in I_N$ with $i\neq j$,
\begin{align}
\label{PScond}
\int_0^t \int_0^s |\nabla K^{\theta,\lambda}_{s-u} ( X^{i,N}_s-X^{i,N}_u)|  \dd u \dd s<\infty 
\end{align}
and if a.s., we have~\eqref{PS} for all $t\geq 0$, all $i\in I_N$.
\end{mydef}

\paragraph{A two-species particle system.}
 We fix $\rho_0 \in \Pp(\R^2)$ and $c_0 \in L^\infty(\R^2)\cap L^1(\R^2)$,
as well as an integer $N\geq 2$ and two real numbers $M>0$ and $\e\in(0,1]$.

\vip
We have $N$ cells indexed by $I_N=\{1,\dots,N\}$ and a varying number of
chemoattractant particles indexed by $\Uu_{N,M}=\Uu^0_{N,M}\cup\Uu^1_N$, where, with 
$L_{N,M}=\lfloor (N-1)M||c_0||_{L^1(\R^2)}\rfloor $,
\begin{gather*}
\text{$\Uu^0_{N,M}:=\{0\}\times\{1,\dots,L_{N,M}\}$ labels the chemoattractant particles present at time $0$,}\\
\text{$a=(a_1,a_2)\in \Uu^1_{N}:=I_N\times\N$ labels
the $a_2$-th chemoattractant particle produced by the cell $a_1$.}
\end{gather*}

We then consider a family of mutually independent random elements:
\vip
\noindent $\bullet$ $X^{k}_0$ for $k \in I_N$, with law $\rho_0$,
\vip
\noindent $\bullet$ $Y^a_0$ for $a \in \Uu^0_{N,M}$, with law $\frac{c_0}{||c_0||_{L^1(R^2)}}$,
\vip
\noindent $\bullet$ $(B^k_t)_{t\geq 0}$ for $k\in I_N$ and $(W^a_t)_{t\geq 0}$ for $a \in \Uu_{N,M}$ some 
$2D$-Brownian motions,
\vip
\noindent $\bullet$ $(\pi^{k,M}_t)_{t\geq 0}$ for $k\in I_N$ 
a Poisson process with rate $\frac M \theta$, we write $\pi^{k,M}_t=\sum_{i\geq 1}\1_{\{T^{k,M}_i\leq t\}}$,
\vip
\noindent $\bullet$ $(\Lambda^a_t)_{t\geq 0}$ for $a \in \Uu_{N,M}$ 
a Poisson process with rate $\frac \lambda \theta$, we write $\Lambda^a_t=\sum_{i\geq 1}\1_{\{S^{a}_i\leq t\}}$.
\vip
We introduce, for each $a\in \Uu_{N,M}$, the birth instant $\tau_a^{N,M}$ and death instant $\zeta_a^{N,M}$
defined by
\begin{gather*}
\text{if $a \in \Uu^0_{N,M}$, then $\tau_a^{N,M}=0$ and  $\zeta_a^{N,M}=S^a_1$,}\\
\text{if $a=(a_1,a_2) \in \Uu^1_{N}$, then $\tau_a^{N,M}=T^{a_1,M}_{a_2}$ and  
$\zeta_a^{N,M}=\min\{S^a_i : i\geq 1, S^a_i\geq \tau_a^{N,M}\}$.}
\end{gather*}
We next consider a family of mollifiers as follows.
\begin{assumption}\label{mol}
Let $\phi$ be a smooth probability density on $\R^2$ with $D\phi, D^2\phi, D^3\phi \in L^1(\R^2)\cap L^\infty(\R^2)$ and satisfying $\phi(-x)=\phi(x)$. For $\e \in (0,1]$ and $x \in \R^2$, we set
\begin{equation*}
\phi_\e(x)=\e^{-1}\phi(\e^{-\frac12}x).
\end{equation*}
\end{assumption}
For $\e\in (0,1]$ fixed, we consider the following particle system:
\begin{equation}\label{troc}
\left\{\begin{array}{lll}
X^{k,N,M,\e}_t\!=\! X^k_0 \!+\! \sqrt 2 B^k_t \! 
+ \! \chi \!\int_0^t (\nu^{k,N,M,\e}_s \!\ast\! \nabla \phi_{\e})(X^{k,N,M,\e}_s)\dd s,
&t\geq 0, \;\; k\in I_N,\\[5pt]
Y^{a,N,M,\e}_t=Y^{a}_0+\sqrt{2\theta^{-1}} W^a_t, & t\in [0,\zeta_a^{N,M}), \;\; a \in \Uu^0_{N,M},\\[5pt]
Y^{a,N,M,\e}_t=X^{a_1,N,M,\e}_{\tau_a^{N,M}}+\sqrt{2\theta^{-1}} (W^a_t-W^a_{\tau_a^{N,M}}), & t\in [\tau_a^{N,M},\zeta_a^{N,M}), 
\;\; a \in \Uu^1_{N},\\[5pt]
\nu^{k,N,M,\e}_t= \frac1{(N-1)M}\sum_{a \in \Uu_{N,M},a_1\neq k} \1_{\{t\in [\tau_a^{N,M},\zeta_a^{N,M})\}}\delta_{Y^{a,N,M,\e}_t}, 
& t\geq 0, \;\; k\in I_N.
\end{array}\right.
\end{equation}
We also introduce
\begin{equation*}
\mu^{N,M,\e}=\frac 1 N \sum_{i \in I_N} \delta_{(X^{i,N,M,\e}_t)_{t\geq 0}}, \;\;
\mu^{N,M,\e}_t=\frac 1 N \sum_{i\in I_N} \delta_{X^{i,N,M,\e}_t} \;\; \text{and} \;\;
\mu^{k,N,M,\e}_t=\frac 1 {N-1} \sum_{i\in I_N,i\neq k} \delta_{X^{i,N,M,\e}_t}.
\end{equation*}

The set of birth and death instants $\Ss=\{\tau_a^{N,M}$, $\zeta_a^{N,M} : a \in \Uu_{N,M}\}$
is well and uniquely defined from the Poisson processes
$((\pi^{k,M}_t)_{t\geq 0} : k\in I_N)$ and $((\Lambda^a_t)_{t \geq 0}:a\in \Uu_{N,M})$. 
It is discrete: it can be written as $\Ss=\{U_k : k\geq 0\}$ with 
$0=U_0<U_1<\dots$. Since $\nabla \phi_\e$ is Lipschitz continuous, one easily
shows the following result, working by induction on the time intervals $[U_i,U_{i+1})$.

\begin{remark}\label{rkeu}
For any  $\rho_0 \in \Pp(\R^2)$, any $c_0 \in L^\infty(\R^2)\cap L^1(\R^2)$, any $\chi>0$, $\lambda>0$, 
$\theta>0$, any integer $N\geq 2$, any $M>0$ and any $\e\in (0,1]$,
the above system has a pathwise unique solution. We say that
$(X^{i,N,M,\e}_t)_{t\geq 0,i\in I_N}$ is a $(\rho_0,c_0,N,M,\e)$-2SKS particle system.
It holds that the family $((X^{i,N,M,\e}_t,\nu^{i,N,M,\e}_t)_{t\geq 0} : i \in I_N)$ is exchangeable.
\end{remark}

\paragraph{Main result.}

We fix $\rho_0 \in \Pp(\R^2)$ and $c_0 \in L^\infty(\R^2)\cap L^1(\R^2)$,
as well as $\chi>0$, $\lambda>0$ and $\theta>0$. We 
consider, for each integer $N\geq 2$, each $M>0$, each $\e\in (0,1]$, the  $(\rho_0,c_0,N,M,\e)$-2SKS system
$(X^{i,N,M,\e}_t)_{t\geq 0,i\in I_N}$ with some $\e\in (0,1]$. Consider also the threshold $\chi^*_\theta>0$ 
defined in Remark~\ref{chistar}.
We endow $C(\R_+,\R^2)$ with the topology of uniform convergence on compact intervals
and $\Pp(C(\R_+,\R^2))$ with the weak convergence topology.

\begin{theorem}\label{main}
Suppose Assumption~\ref{mol} and that $\chi<\chi^*_\theta$. Fix $\sigma\in (0,\frac 12]$.

\vip

(i) Set
\begin{equation}\label{condie1}
\Aa_\sigma=\{(N,M,\e) : N\geq 2, M>0, \e\in (0,1], N M\e^{3+\sigma}\geq 1, N M^{1+\sigma} \e^{2+\sigma} \geq 1\}.
\end{equation}
The family 
$((X^{1,N,M,\e}_t)_{t\geq 0} : (N,M,\e)\in \Aa_\sigma)$ is tight in $C(\R_+,\R^2)$.

\vip

(ii) Fix $N\geq 2$ and consider $(M_n)_{n\geq 1}$ valued in $(0,\infty)$ and $(\e_n)_{n\geq1}$ valued in 
$(0,1]$ such that 
\begin{equation}\label{condie2}
\lim_{n\to \infty} M_n= \infty, \quad \lim_{n\to \infty} \e_n=0 \quad \text{and} \quad \lim_{n\to\infty}
M_n\e_n^{3+\sigma}=\infty.
\end{equation}
Then the family $((X^{k,N,M_n,\e_n}_t)_{t\geq 0, k \in I_N} : n\geq 1)$ is tight in $C(\R_+,(\R^2)^N)$
and any limit point $((X^{k,N}_t)_{t\geq 0, k \in I_N}$ as $M\to \infty$ is a $(\rho_0,c_0,N)$-1SKS system.

\vip

(iii) Consider  $(M_N)_{N\geq 2}$ valued in $(0,\infty)$ and $(\e_N)_{N\geq2}$ valued in $(0,1]$ such that 
\begin{equation}\label{condie3}
\lim_{N\to \infty} \e_{N}=0, \quad \lim_{N\to \infty}  N M_N\e_N^{3+\sigma}=\infty\quad \text{and} 
\quad \lim_{N\to \infty} N M_N^{1+\sigma}\e_N^{2+\sigma}=\infty.
\end{equation}
Then the family $(\mu^{N,{M_N}} : N \geq 2)$ is tight in $\Pp(C(\R_+,\R^2))$ and any (possibly random) 
limit $\mu$ a.s. solves $\mathcal{(MP)}(\rho_0,c_0)$.
\end{theorem} 

As a consequence of the last point, there exists, under the same assumptions, a subsequence $N_n\to \infty$ 
such that $(\mu^{N_n,{M_{N_n}}}_t)_{t\geq 0}$ converges in law to some (possibly random)
$(\mu_t)_{t\geq 0}$ and, setting $\nu_t=b_t^{c_0,\theta,\lambda}+ \int_0^t (K^{\theta,\lambda}_{t-s} \ast \mu_s )\,\dd s$, it holds that $(\mu_t,\nu_t)_{t\geq 0}$ is a.s. a weak solution of the Keller-Segel equation~\eqref{EDP_KS} in the sense of~\cite[Definition~1]{FT1}.
See the introduction for comments about the threshold $\chi^*_\theta$, about the assumptions on the initial conditions and about~\eqref{condie2} and~\eqref{condie3}. 

\section{Preliminaries}
\label{spre}

In this section, we present a crucial functional inequality, we then introduce the threshold for the parameter
$\chi$ and we finally derive an S.D.E. for the chemoattractant empirical measure.

\paragraph{Functional inequality.} 
The following functional inequality is proved in~\cite{FT1}.

\begin{lemma}
\label{lemma:FI}
Let $b>a>0$ and $t>0$. For any measurable function $f:[0, t] \to \R_+  $, we have 
$$
\int_0^t \frac{\dd s}{(t-s+ f(s))^{1+a}}\leq \kappa(a,b) 
\Big(\int_0^t \frac{\dd s}{(t-s+ f(s))^{1+b}}\Big)^{\frac{a}{b}},\quad \hbox{where} \quad
\kappa(a,b)=\frac{a+1}{a}\Big[\frac{b}{b+1}\Big]^{\frac a b}.
$$
\end{lemma}

\paragraph{The threshold.}

We introduce the same constants as in~\cite{FT1}: for $\alpha>0$, $\beta>0$, $\theta>0$ and 
$\gamma>\frac 32$,
\begin{gather}\label{eq:C0}
C_0(\beta):=\sup_{u\ge 0} \sqrt u (1+\beta u) ^{\frac 32} e^{-u}, \qquad
C_1(\alpha,\gamma):=(\gamma-1)(1-4\alpha(\gamma-1)),\\
\label{eq:C2}
C_2(\theta,\alpha,\gamma):= \frac{\sqrt{\alpha\theta}(\gamma-1)}{2\pi}C_0\Big(\frac{4\alpha}{\theta}\Big) 
\kappa\Big(\frac{1}{2}, \gamma-1\Big)\kappa\Big(\gamma-\frac32,\gamma-1\Big).
\end{gather}
We now summarize~\cite[Section~6]{FT1} in the case where $c_0 \in L^\infty(\R^2)$.

\begin{remark}\label{chistar}
For $\theta>0$, set 
$$
\chi^*_\theta=\sup\Big\{ \chi_{\theta,\alpha,\gamma} : \gamma \in \Big(\frac32,2\Big), 
\alpha \in \Big(0,\frac1{4(\gamma-1)}\Big) \Big\},
$$
where
$$
\chi_{\theta,\alpha,\gamma}= \sup\Big\{\chi>0 :\chi C_2(\theta,\alpha,\gamma) + 
\Big(\frac{\chi\sqrt\theta C_0\big(\frac{4\alpha}\theta\big) \kappa\big(\frac{1}{2}, \gamma-1\big)}
{4\pi\sqrt\alpha(4-2\gamma)}\Big)^{2(\gamma-1)} < C_1(\alpha,\gamma) \Big\}.
$$
If $\chi \in (0,\chi^*_\theta)$, we can find $\alpha>0$ and  $\gamma \in (\frac32,2)$ such that
\begin{equation}\label{condistar}
C_1(\alpha,\gamma)\!>\!0,\;\; C_1(\alpha,\gamma)\!>\!\chi C_2(\theta,\alpha,\gamma) \;\;\text{and}\;\;
(4-2\gamma)> \frac{\chi\sqrt{\theta} C_0\big(\frac{4\alpha}{\theta}\big) 
\kappa\big(\frac{1}{2}, \gamma-1\big)}{4\pi\sqrt{\alpha}[C_1(\alpha,\gamma) 
-\chi C_2(\theta,\alpha,\gamma)]^{\frac 1{2(\gamma-1)}}}.
\end{equation}
\end{remark}

The set of conditions~\eqref{condistar} will be used in Proposition~\ref{prop:mc}.

\paragraph{An S.D.E. formula for the chemoattractant particles.}
We establish here a formula concerning the chemoattractant particles concentration.
For $\varphi:\R_+\times\R^2 \to \R$, for $\mu$ a measure on $\R^2$ and for $t\geq 0$, we write 
$\langle \mu,\varphi(t,\cdot)\rangle=\int_{\R^2}\varphi(t,y)\mu(\dd y)$.

\begin{lemma}\label{itonu} 
Fix $\rho_0 \in \Pp(\R^2)$, $c_0 \in L^\infty(\R^2)\cap L^1(\R^2)$, as well as $\chi>0$, $\lambda>0$, 
$\theta>0$, $N\geq 2$, $M>0$ and $\e\in (0,1]$. Consider the $(\rho_0,c_0,N,M,\e)$-2SKS system
with all the associated random objects. For all $\varphi \in C^{2}(\R_+\times\R^d)$, all $t\geq 0$,
\begin{align*}
\langle \nu^{1,N,M,\e}_t,\varphi(t,\cdot)\rangle=& \langle \nu^{1,N,M,\e}_0,\varphi(0,\cdot)\rangle
+\frac 1 \theta \int_0^t \langle \mu^{1,N,M,\e}_s,\varphi(s,\cdot)\rangle\dd s\\
&+ \int_0^t \Big\langle  \nu^{1,N,M,\e}_s, \partial_t\varphi(s,\cdot)+\frac1\theta(\Delta \varphi(s,\cdot)
-\lambda \varphi(s,\cdot))\Big\rangle \dd s\\
&+\frac1{(N-1)M}\sum_{a \in \Uu_{N,M},a_1\neq 1}\sqrt{\frac 2 \theta}\int_0^t \1_{\{s\in [\tau_a^{N,M},\zeta_a^{N,M})\}} 
\nabla\varphi(s,Y^{a,N,M,\e}_s) \dd W^a_s\\
&+\frac1{(N-1)M}\sum_{i\in I_N,i\neq 1}\int_0^t\varphi(s,X^{i,N,M,\e}_s) \dd \tilde \pi_s^{i,M}\\
&-\frac1{(N-1)M}\sum_{a \in \Uu_{N,M},a_1\neq 1}\int_0^t \1_{\{s\in (\tau_a^{N,M},\zeta_a^{N,M}]\}}
\varphi(s,Y^{a,N,M,\e}_\sm) \dd \tilde \Lambda^a_s,
\end{align*}
where $\tilde \pi_t^{i,M}=\pi^{i,M}_t-\frac M\theta t$ and $\tilde \Lambda^a_t= \Lambda^a_t- \frac\lambda \theta t$.
\end{lemma}

\begin{proof}
We remove the superscripts $N,M$ and $N,M,\e$ for simplicity.
Multiplying the target formula by $N(M-1)$, {\it decompensating} and using the expressions 
of $\nu^1_t$ and $\mu^1_t$, we have to show that $H_t=H_0+I_t^1+I^2_t+I^3_t$, where
\begin{align*}
H_t=&\sum_{a\in \Uu_{N,M},a_1\neq 1} \1_{\{t \in [\tau_a,\zeta_a)\}}\varphi(t,Y^a_t),\\
I^1_t=& \sum_{a\in \Uu_{N,M},a_1\neq 1} \!\!\Big( \int_0^t \!\! \1_{\{s\in [\tau_a,\zeta_a)\}}
\Big(\partial_t\varphi(s,Y^a_s)\dd s+\frac1\theta
\Delta \varphi(s,Y^a_s) \dd s +\sqrt{\frac 2 \theta}\nabla\varphi(s,Y^{a}_s) \dd W^a_s\Big),\\
I^2_t=&\sum_{i\in I_N,i\neq 1}\int_0^t\varphi(s,X^{i}_s) \dd \pi_s^{i,M},\\
I^3_t=&-\sum_{a \in \Uu_{N,M},a_1\neq 1}\int_0^t \1_{\{s\in (\tau_a,\zeta_a]\}}
\varphi(s,Y^{a}_\sm) \dd \Lambda^a_s.
\end{align*}

We consider the discrete sets $\Ss_B=\{\tau_a : a \in \Uu_N^1\}$ of birth instants and
$\Ss_D= \{\zeta_a : a \in \Uu_{N,M}\}$ of death instants. We order all these instants chronologically,
i.e. we write $\Ss_B\cup \Ss_D=\{U_k : k\geq1\}$, with $0<U_1<U_2<\cdots$. We also set $U_0=0$.
\vip
Recall~\eqref{troc}. We have $\tau_a\geq U_1$ for all $a \in \Uu_N^1$, $\zeta_a\geq U_1$ for all 
$a \in \Uu_{N,M}$ and $Y^a_t=Y^a_0+\sqrt{2\theta^{-1}}W^a_t$ 
for all $a \in \Uu_{N,M}^0$, all $t\in [0,U_1)$. By Itô's formula, for
$t \in [0,U_1)$, 
$$
H_t=\!\!\!\!\sum_{a\in \Uu_{N,M}^0}\varphi(t,Y^a_t)=\!\!\!\!\sum_{a\in \Uu_{N,M}^0} \Big(\varphi(0,Y^a_0) 
+ \int_0^t \!\! \Big(\partial_t\varphi(s,Y^a_s)\dd s+\frac1\theta
\Delta \varphi(s,Y^a_s) \dd s+ \sqrt{\frac 2 \theta}
\nabla\varphi(s,Y^{a}_s) \dd W^a_s\Big).
$$
Since $s\in [\tau_a,\zeta_a)$ for all $s\in [0,U_1)$, all $a \in \Uu^0_{N,M}$, since $\tau_a\geq U_1$
for all $a \in \Uu^1_N$ and since $I^2_t=I^3_t=0$ for all $t\in [0,U_1)$,
we conclude that $H_t=H_0+I_t^1+I^2_t+I^3_t$ for $t\in [0,U_1)$.

\vip

We now assume that $H_t=H_0+I_t^1+I^2_t+I^3_t$ for $t\in [U_{n-1},U_n)$ and we prove that the same formula
holds on $[U_n,U_{n+1})$. If first $U_n \in \Ss_B$, i.e. if $U_n=\tau_b$ for some $b=(b_1,b_2) \in \Uu^1_N$,
then
$$
H_{U_n}=H_{U_n-}+\varphi(U_n,X^{b_1}_{U_n})=H_{U_n-}+\int_{\{s=U_n\}}\varphi(s,X^{b_1}_{s}) \dd\pi^{b_1,M}_s
=H_{U_n-}+\Delta I^2_{U_n},
$$
where $\Delta I^2_t=I^2_{t}-I^2_{t-}$. We used that
a.s., $\Delta \pi^{i,M}_{U_n}=0$ for all $i\neq b_1$. For the same reason, $\Delta I^3_{U_n}=0$ and, of course,
$\Delta I^1_{U_n}=0$. Since $H_{U_n-}=H_0+I^1_{U_n-}+I^2_{U_n-}+I^3_{U_n-}$ by induction assumption,
\begin{equation}\label{updc}
H_{U_n}=H_{U_n-}+\Delta I^2_{U_n}=H_0+I^1_{U_n}+I^2_{U_n}+I^3_{U_n}.
\end{equation}
If next $U_n \in \Ss_D$, i.e. if $U_n=\zeta_b$ for some $b \in \Uu_{N,M}$ (with necessarily
$\tau_b <U_n)$, then
$$
H_{U_n}=H_{U_n-}-\varphi(U_n,Y^b_{U_n-})=H_{U_n-}-\int_{\{s=U_n\}}\1_{\{s\in (\tau_b,\zeta_b]\}}\varphi(s,Y^{b}_\sm) \dd\Lambda^{b}_s
=H_{U_n-}+\Delta I^3_{U_n},
$$
since a.s., $\Delta\Lambda^a_{U_n}=0$ for all $a\neq b$. But $\Delta I^1_{U_n}=\Delta I^2_{U_n}=0$ a.s., so that
we also have~\eqref{updc}. 
\vip

Now we fix $t\in (U_n,U_{n+1})$ and use Itô's formula to write
\begin{align*}
H_t=&\!\!\!\!\sum_{a\in \Uu_{N,M},a_1\neq 1}\1_{\{t\in [\tau_a,\zeta_a)\}}\varphi(t,Y^a_t)\\
=&\!\!\!\!\sum_{a\in \Uu_{N,M},a_1\neq 1} \!\!\!\!\1_{\{t\in [\tau_a,\zeta_a)\}}\Big(\varphi(U_n,Y^a_{U_n}) 
\!+\! \int_{U_n}^t \!\! \Big(\partial_t\varphi(s,Y^a_s)\!+\!\frac1\theta
\Delta \varphi(s,Y^a_s)\Big) \dd s+
\sqrt{\frac 2 \theta}\int_{U_n}^t\!\!
\nabla\varphi(s,Y^{a}_s) \dd W^a_s\Big)\\
=&H_{U_n}+ I^1_t-I^1_{U_n},
\end{align*}
since for each $a \in \Uu_{N,M}$ such that $\tau_a\leq t<\zeta_a$, we have $\tau_a\leq U_n<U_{n+1}\leq \zeta_a$, 
whence $s\in [\tau_a,\zeta_a)$ for all $s \in (U_n,t)$. Since moreover
$I^2_t=I^2_{U_n}$ and $I^3_t=I^3_{U_n}$, we end with 
$$
H_t=H_{U_n}+I^1_t-I^2_{U_n}+I^2_t-I^1_{U_n}+I^3_t-I^3_{U_n}.
$$
By~\eqref{updc}, we end with $H_t=H_0+I^1_t+I^2_t+I^3_t$ as desired.
\end{proof}

\section{From chemotactic attraction to path-guided attraction in 2SKS}
\label{fcat}

In this section, we fix $\rho_0 \in \Pp(\R^2)$ and $c_0 \in L^\infty(\R^2)\cap L^1(\R^2)$,
as well as $\chi>0$, $\lambda>0$ and $\theta>0$. We consider, for each integer $N\geq 2$, each $M>0$, 
each $\e\in (0,1]$, the $(\rho_0,c_0,N,M,\e)$-2SKS particle system $(X^{i,N,M,\e}_t)_{t\geq 0, i \in I_N}$
together with all the associated random objects.
Our goal is to prove the following proposition which shows that
the dynamics of the 2SKS system is not far from that of the 1SKS system: The drift of the cells in~\eqref{troc}
resembles the drift in~\eqref{PS}.

\begin{proposition}\label{prel}
Consider the process $(U^{1,N,M,\e}_t)_{t\geq 0}$ defined by
\begin{align*}
U^{1,N,M,\e}_t=&(\nu^{1,N,M,\e}_t\ast \nabla \phi_{\e})(X^{1,N,M,\e}_t)-(\nabla \phi_{\e}\ast b^{c_0,\theta,\lambda}_t)(X^{1,N,M,\e}_t)\\
&-\frac1{N-1} \sum_{j \in I_N,j\neq 1}\int_0^t (\nabla \phi_\e\ast K^{\theta,\lambda}_{t-s})(X^{1,N,M,\e}_t-X^{j,N,M,\e}_s)\dd s.
\end{align*}
For any $\sigma\in (0,\frac12]$, any $T>0$, there exists a constant
$A_{\sigma,T}$ (also depending on $\lambda,\theta,c_0$) such that for all $N\geq 2$, all $M>0$, all $\e\in (0,1]$,
$$
\sup_{t\in [0,T]}\E[|U^{1,N,M,\e}_t|^2] \leq \frac{A_{\sigma,T}}{(NM \e^{\frac12})^{2}}
+\frac{A_{\sigma,T}}{NM \e^{3+\sigma}}+\frac{A_{\sigma,T}}{NM^{1+\sigma}\e^{2+\sigma}}.
$$
\end{proposition}

\begin{proof}
We fix $\sigma\in (0,\frac12]$ and we set $r= \frac2{1-\sigma}\in (2,4]$. We remove the superscripts $N,M$ and $N,M,\e$ for simplicity.

\vip
{\it Step 1.} We fix $t>0$ and $x\in \R^2$ and we apply Lemma~\ref{itonu} (at time $t$) to the function
$$
\varphi^\e_{t,x}(s,y)=\theta(\nabla \phi_\e \ast K^{\theta,\lambda}_{t-s})(x-y), 
$$
of which the properties are given in Lemma~\ref{tri}.
Observe that $\langle \nu^1_t,\varphi^\e_{t,x}(t,\cdot)\rangle=
(\nu^1_t\ast \nabla \phi_\e)(x)$ because $\varphi^\e_{t,x}(t,y)=\nabla \phi_\e(x-y)$, see~\eqref{tri2}
and that $\partial_s \varphi^\e_{t,x}+\frac1\theta(\Delta_y\varphi^\e_{t,x}-\lambda\varphi^\e_{t,x})=0$, see 
also~\eqref{tri2}. We find
\begin{align}\label{tbua}
(\nu^1_t\ast \nabla \phi_\e)(x)=& I_t(x)+ J_t(x)+M^1_t(x)+M^2_t(x)+M^3_t(x),
\end{align}
where
\begin{align*}
I_t(x)=&\langle \nu^1_0,\varphi^\e_{t,x}(0,\cdot)\rangle,\\
J_t(x)=&\frac1\theta \int_0^t\langle \mu^1_s,\varphi^\e_{t,x}(s,\cdot)\rangle \dd s
=\frac1{N-1}\sum_{j\in I_N,j\neq 1}\int_0^t  (\nabla \phi_{\e}\ast K^{\theta,\lambda}_{t-s})(x-X^{j}_s)\dd s,\\
M^1_t(x)=&\frac1{(N-1)M}\sum_{a \in \Uu_{N,M},a_1\neq 1}\sqrt{\frac 2 \theta}\int_0^t \1_{\{s\in [\tau_a,\zeta_a)\}} 
\nabla_y\varphi^\e_{t,x}(s,Y^{a}_s) \dd W^a_s,\\
M^2_t(x)=&\frac1{(N-1)M}\sum_{i\in I_N,i\neq 1}\int_0^t \varphi^\e_{t,x}(s,X^{i}_s) \dd \tilde \pi_s^{i,M},\\
M^3_t(x)=&-\frac1{(N-1)M}\sum_{a \in \Uu_{N,M},a_1\neq 1}\int_0^t \1_{\{s\in (\tau_a,\zeta_a]\}} 
\varphi^\e_{t,x}(s,Y^{a}_\sm) \dd \tilde \Lambda^a_s.
\end{align*}
Applying this formula with $x=X^1_t$, we conclude that
$$
U^1_t=\Delta_t(X^1_t)+M^1_t(X^1_t)+M^2_t(X^1_t)+M^3_t(X^1_t), \quad \text{where} \quad
\Delta_t(x)= I_t(x)- (\nabla \phi_\e\ast b^{c_0,\theta,\lambda}_t)(x).
$$
We next write $\Delta_t(x)=\Delta^1_t(x)+\Delta^2_t(x)$, where
\begin{align*}
\Delta^1_t(x)=&\Big(\frac{L_{N,M}}{(N-1)M||c_0||_{L^1(\R^2)}}-1\Big) \nabla (\phi_\e\ast b^{c_0,\theta,\lambda}_t)(x),\\
\Delta^2_t(x)=&I_t(x) 
- \frac{L_{N,M}}{(N-1)M||c_0||_{L^1(\R^2)}} \nabla (\phi_\e\ast b^{c_0,\theta,\lambda}_t)(x).
\end{align*}
To complete the proof, we will show that there exists a constant
$A_{\sigma,T}$ such that
\begin{gather*}
\sup_{t\in [0,T]} ||\Delta^1_t||_{L^\infty(\R^2)}^2 \leq \frac{A_{\sigma,T}}{(NM\e^{\frac12})^2},\\
\sup_{t\in[0,T]}\E\Big[\sum_{k=1}^3||M^k_t||_{L^\infty(\R^2)}^2+||\Delta^2_t||_{L^\infty(\R^2)}^2\Big]
\leq \frac{A_{\sigma,T}}{NM\e^{3+\sigma}}+  \frac{A_{\sigma,T}}{NM^{1+\sigma}\e^{2+\sigma}} .
\end{gather*}
To study $||M^k_t||_{L^\infty(\R^2)}$ and $||\Delta^2_t||_{L^\infty(\R^2)}$, we will use a variant of the methods of~\cite{ORT}, see also~\cite{CRT}. Recall the following Gagliardo-Nirenberg inequality: 
For any $r>2$, there exists $A_r$ such that for all $u:\R^2 \to \R$,
\begin{equation}\label{gn}
||u||_\infty \leq A_r ||\nabla u||_{L^r(\R^2)}^{\frac 2 r}||u||_{L^r(\R^2)}^{1-\frac 2 r},
\end{equation}
see e.g. Cazenave~\cite[Theorem~1.3.7]{Caz}
with $N=2$, $j=0$, $m=1$, $p=\infty$, $q=r>2$, $r=r$ and $a=\frac 2 r$.
\vip

{\it Step 2.} Recalling that $L_{N,M}=\lfloor (N-1)M||c_0||_{L^1(\R^2)}\rfloor$ and that
$N-1\geq \frac N2$, there exists a constant $A$ (depending on $c_0,\theta$) such that, see~\eqref{rmu},
$$
||\Delta_t^1||_{L^\infty(\R^2)}\leq \frac{1}{(N-1)M||c_0||_{L^1(\R^2)}} 
||\nabla(\phi_\e \ast b_t^{c_0,\theta,\lambda})||_{L^\infty(\R^2)} \leq 
\frac{A}{NM\e^{\frac 1 2}}.
$$

\vip

{\it Step 3.} Here we study $||M^1_t||_{L^\infty(\R^2)}$.
Recall that $r>2$. By independence of the Brownian motions $W^a$, by the 
Burkholder-Davies-Gundy inequality and since $N-1\geq \frac N2$,
\begin{align*}
\E[|\nabla M^1_t(x)|^r] \leq& \frac{A_r}{(MN)^r}\E\Big[\Big(\sum_{a \in \Uu_{N,M},a_1\neq 1}\int_0^t
|\nabla_x\nabla_y\varphi^\e_{t,x}(s,Y^{a}_s)|^2\1_{\{s \in [\tau_a,\zeta_a)} \dd s\Big)^{\frac r2}\Big]\\
 \leq& \frac{A_r}{(MN)^r}\E\Big[\Big(\sum_{a \in \Uu_{N,M}}\1_{\{\tau_a\leq t\}} \int_0^t 
|\nabla_x\nabla_y\varphi^\e_{t,x}(s,Y^{a}_s)|^2\dd s\Big)^{\frac r2}\Big]\\
\leq & \frac{A_{r,T}}{(MN)^r}\E\Big[\Big(\sum_{a \in \Uu_{N,M}} \1_{\{\tau_a\leq t\}}\int_0^t
|\nabla_x\nabla_y\varphi^\e_{t,x}(s,Y^{a}_s)|^r \dd s\Big) L_t^{\frac r2-1}\Big],
\end{align*}
where $L_t=\sum_{a \in \Uu_{N,M}}\1_{\{\tau_a \leq t\}}$.
Integrating in $x$, we deduce from~\eqref{estp2} that for all $t\in [0,T]$,
$$
\E[||\nabla M^1_t||_{L^r(\R^2)}^r]\leq \frac{A_{r,T}}{(MN)^r\e^{\frac{5r}2-2}}\E[L_T^{\frac r2}]
\leq \frac{A_{r,T}}{(MN)^{\frac r 2}\e^{\frac{5r}2-2}}.
$$
We used that $L_T=L_{N,M}+Z_T$, with $L_{N,M}=\sum_{a \in \Uu_{N,M}^0}\1_{\{\tau_a \leq t\}}=
\lfloor (N-1)M||c_0||_{L^1(\R^2)}\rfloor$ and where 
$Z_T=\sum_{a \in \Uu_{N}^1}\1_{\{\tau_a \leq t\}}$ follows 
a Poisson$(\frac{TNM}\theta)$-distribution, whence  $\E[L_T^{\frac r 2}]\leq A_{r,T} (NM)^{\frac r2}$.
By the very same arguments, using~\eqref{estp1} instead of~\eqref{estp2},
$$
\E[||M^1_t||_{L^r(\R^2)}^r] \leq \frac{A_{r,T}}{(MN)^{\frac r 2}\e^{2r-2}}.
$$
Using now~\eqref{gn} and the Hölder inequality with $p=\frac {r^2}{4}$, $p'=\frac {r^2}{2(r-2)}$ and 
$p''=\frac {r}{r-2}$ (recall that $r>2$), we get, for all $t\in [0,T]$,
\begin{align*}
\E[||M^1_t||_{L^\infty(\R^2)}^2]\leq& A_r\E\Big[||\nabla M^1_t||_{L^r(\R^2)}^{\frac{4}r}
||M^1_t||_{L^r(\R^2)}^{2-\frac{4}r} \times 1\Big]\\
\leq& A_r \E\Big[||\nabla M^1_t||_{L^r(\R^2)}^r\Big]^{\frac{4}{r^2}}
\E\Big[||M^1_t||_{L^r(\R^2)}^r\Big]^{\frac{2(r-2)}{r^2}}.
\end{align*}
Using the previous estimates, this precisely gives (recall that $r=\frac2{1-\sigma}$)
\begin{equation*}
\sup_{t \in [0,T]}\E[||M^1_t||_{L^\infty(\R^2)}^2]\leq \frac{A_{r,T}}{MN\e^{4-\frac 2 r}}
=\frac{A_{\sigma,T}}{MN\e^{3+\sigma}}.
\end{equation*}

{\it Step 4.} Here we prove that for $a\in [1,2]$, there exists $A_a>0$ such that for $(N_t)_{t\geq 0}$ 
a Poisson process with parameter $\lambda>0$, for 
$(\alpha_t)_{t\geq 0}$ a nonnegative predictable process and for $H_t=\int_0^t \alpha_s\dd N_s$,
$$
\E[H_t^a]\leq A_a \lambda (1+(\lambda t)^{a-1})\int_0^t \E[\alpha_s^a]\dd s.
$$
To this end, we write $H_t=D_t+M_t$, where $D_t=\lambda\int_0^t \alpha_s \dd s$
and $M_t=\int_0^t \alpha_s \dd \tilde N_s$. By the Hölder inequality, 
$\E[D_t^a]\leq \lambda^a t^{a-1} \int_0^t \E[\alpha_s^a] \dd s$. By the Burkholder-Davies-Gundy inequality,
$$
\E[|M_t|^a]\leq A_a \E\Big[\Big(\int_0^t \alpha_s^2 \dd N_s\Big)^{\frac a2}\Big]
\leq A_a \E\Big[\int_0^t \alpha_s^a \dd N_s\Big]=A_a \lambda \int_0^t \E[\alpha_s^a] \dd s.
$$
We used that $\frac a2 \in [0,1]$ for the second inequality. The conclusion follows.

\vip

{\it Step 5.} We now study $||M^2_t||_{L^\infty(\R^2)}$. Since $N-1\geq \frac N2$, 
by independence of the Poisson measures $\pi^{i,M}$ and
by the Burkholder-Davies-Gundy inequality,
\begin{align*}
\E[|\nabla M^2_t(x)|^r]\leq& \frac{A_r}{(NM)^r} \E\Big[\Big(\sum_{i \in I_N, i\neq 1}\int_0^t
|\nabla_x \varphi^\e_{t,x}(s,X^i_s)|^2 \dd \pi^{i,M}_s \Big)^{\frac r2}\Big]\\
\leq &  \frac{A_r N^{\frac r 2 -1}}{(NM)^r} \E\Big[\sum_{i \in I_N} \Big(\int_0^t
|\nabla_x \varphi^\e_{t,x}(s,X^i_s)|^2 \dd \pi^{i,M}_s \Big)^{\frac r2}\Big]\\
=&  \frac{A_r N^{\frac r 2}}{(NM)^r} \E\Big[\Big(\int_0^t
|\nabla_x \varphi^\e_{t,x}(s,X^1_s)|^2 \dd \pi^{1,M}_s \Big)^{\frac r2}\Big]
\end{align*}
by exchangeability. We now use Step~4, recalling that $r \in (2,4)$, that $(\pi^{1,M}_t)_{t\geq 0}$ has for parameter
$\frac M \theta$ and noting that $\frac M \theta(1+(\frac{Mt}\theta)^{\frac r2-1}))\leq A_T (M+M^{\frac r 2})$
for $t\in [0,T]$.
We get, for all $t\in [0,T]$,
$$
\E[|\nabla M^2_t(x)|^r]\leq \frac{A_{r,T} N^{\frac r 2} (M+M^{\frac r 2})}{(NM)^r} \E\Big[\int_0^t
|\nabla_x \varphi^\e_{t,x}(s,X^1_s)|^r \dd s\Big].
$$
Integrating in $x$ and using~\eqref{estp1}, we find
$$
\E[||\nabla M^2_t||_{L^r(\R^2)}^r] \leq \frac{A_{r,T} N^{\frac r 2}(M+M^{\frac r 2})}{(NM)^{r} \e^{2r-2}}
\leq  \frac{A_{r,T} }{(N[M\land M^{2-\frac2r} ])^{\frac r2} \e^{2r-2}}
$$
We check similarly that 
$$
\E[||M^2_t||_{L^r(\R^2)}^r] \leq \frac{A_{r,T}}{(N[M\land M^{2-\frac2r} ])^{\frac r 2} \e^{\frac {3r}2-2}}.
$$
Exactly as in Step~3, 
\begin{align*}
\E[||M^2_t||_{L^\infty(\R^2)}^2]\leq A_{r} \E\Big[||\nabla M^2_t||_{L^r(\R^2)}^r\Big]^{\frac{4}{r^2}}
\E\Big[||M^2_t||_{L^r(\R^2)}^r\Big]^{\frac{2(r-2)}{r^2}}.
\end{align*}
As a conclusion,
\begin{equation*}
\sup_{t\in [0,T]}\E[||M^2_t||_{L^\infty(\R^2)}^2]\leq \frac{A_{r,T}}{N[M\land M^{2-\frac2r} ]\e^{3-\frac 2 r}}
\leq \frac{A_{r,T}}{NM\e^{3-\frac 2 r}}+ \frac{A_{r,T}}{NM^{2-\frac 2r}\e^{3-\frac 2r}}.
\end{equation*}
Since $r=\frac 2{1-\sigma}$, this gives
\begin{equation*}
\sup_{t\in [0,T]}\E[||M^2_t||_{L^\infty(\R^2)}^2]
\leq \frac{A_{\sigma,T}}{NM\e^{2+\sigma}}+ \frac{A_{\sigma,T}}{NM^{1+\sigma}\e^{2+\sigma}}
\leq \frac{A_{\sigma,T}}{NM\e^{3+\sigma}}+ \frac{A_{\sigma,T}}{NM^{1+\sigma}\e^{2+\sigma}}.
\end{equation*}

{\it Step 6.} Here we study $||M^3_t||_{L^\infty(\R^2)}$. By the Burkholder-Davies-Gundy inequality,
\begin{align*}
\E[|\nabla M^3_t(x)&|^r]\leq \frac{A_r}{(NM)^r} \E\Big[\Big(\sum_{a \in \Uu_{N,M}}\int_0^t
|\nabla_x \varphi^\e_{t,x}(s,Y^a_\sm)|^2 \1_{\{s \in (\tau_a,\zeta_a]\}} \dd \Lambda^a_s \Big)^{\frac r2}\Big]\\
\leq &  \frac{A_r}{(NM)^r} \E\Big[\Big(\sum_{a \in \Uu_{N,M}}\int_0^t
\1_{\{s \in (\tau_a,\zeta_a]\}}\dd \Lambda^a_s \Big)^{\frac r2-1}  \sum_{a \in \Uu_{N,M}}\int_0^t
|\nabla_x \varphi^\e_{t,x}(s,Y^a_\sm)|^r \1_{\{s \in (\tau_a,\zeta_a]\}}\dd \Lambda^a_s \Big].
\end{align*}
Integrating in $x$ thanks to~\eqref{estp0}, we find
\begin{align*}
\E[||\nabla M^3_t||_{L^r(\R^2)}^r]\leq& \frac{A_r}{(NM)^r} \E\Big[\Big(\sum_{a \in \Uu_{N,M}}\int_0^t
\1_{\{s \in (\tau_a,\zeta_a]\}}\dd \Lambda^a_s \Big)^{\frac r2-1}  \sum_{a \in \Uu_{N,M}}\int_0^t
\frac {\1_{\{s \in (\tau_a,\zeta_a]\}}}{(\e+(t-s))^{2r-1}} \dd \Lambda^a_s \Big].
\end{align*}
By definition of $\tau_a$ and $\zeta_a$, we have 
$\int_0^t\1_{\{s \in (\tau_a,\zeta_a]\}}\dd \Lambda^a_s\leq \1_{\{\tau_a\leq t\}}$. Thus
\begin{align*}
\E[||\nabla M^3_t||_{L^r(\R^2)}^r]\leq& \frac{A_r}{(NM)^r} \E\Big[\Big(\sum_{a \in \Uu_{N,M}} \1_{\{\tau_a \leq t\}} 
\Big)^{\frac r2-1}  
\sum_{a \in \Uu_{N,M}}\1_{\{\tau_a\leq t\}} \int_0^t \frac {1}{(\e+(t-s))^{2r-1}} \dd \Lambda^a_s \Big].
\end{align*}
Using now that $((\Lambda^a_t)_{t\geq 0} : a \in \Uu_{N,M})$ is independent of $(\tau_a :  a \in \Uu_{N,M})$
and that $(\Lambda^a_t)_{t\geq 0}$ has for parameter $\frac\lambda\theta$, we get
\begin{align*}
\E[||\nabla M^3_t||_{L^r(\R^2)}^r]\leq& \frac{A_{r}}{(NM)^r} 
\E\Big[\Big(\sum_{a \in \Uu_{N,M}} \1_{\{\tau_a\leq t\}} \Big)^{\frac r2-1}\sum_{a \in \Uu_{N,M}}\1_{\{\tau_a\leq t\}} 
\int_0^t \frac {\dd s}{(\e+(t-s))^{2r-1}} \Big].
\end{align*}
Thus for all $t\in [0,T]$, since $2r-1>1$,
\begin{align*}
\E[||\nabla M^3_t||_{L^r(\R^2)}^r]\leq& \frac{A_r}{(NM)^r\e^{2r-2}} 
\E\Big[\Big(\sum_{a \in \Uu_{N,M}} \1_{\{\tau_a\leq t\}} \Big)^{\frac r2}\Big] \leq 
\frac{A_{r,T}}{(NM)^{\frac r2}\e^{2r-2}}, 
\end{align*}
since we have already seen in Step~3 that 
$\E[(\sum_{a \in \Uu_{N,M}} \1_{\{\tau_a\leq t\}})^{\frac r2}]\leq A_{r,T} (MN)^{\frac r 2}$. One finds similarly
\begin{align*}
\E[||M^3_t||_{L^r(\R^2)}^r]\leq  \frac{A_{r,T}} {(NM)^{\frac r2}\e^{\frac{3r}2-2}}, 
\end{align*}
and we conclude as in Step~5 (replacing $[M \land M^{2-\frac2r}]$ by $M$) that
\begin{equation*}
\sup_{t\in [0,T]}\E[||M^3_t||_{L^\infty(\R^2)}^2]\leq \frac{A_{r,T}}{NM\e^{3-\frac 2 r}}
= \frac{A_{\sigma,T}}{NM\e^{2+\sigma}}\leq\frac{A_{\sigma,T}}{NM\e^{3+\sigma}}.
\end{equation*}

{\it Step 7.} We finally study $||\Delta_t^2||_{L^\infty(\R^2)}$. Since
$Y^a_0\sim ||c_0||^{-1}_{L^1(\R^2)}c_0(y)\dd y$ for all $a \in \Uu^0_{N,M}$ and
since $\varphi^\e_{t,x}(0,y)= \theta (\nabla \phi_\e \ast K_t^{\theta,\lambda})(x-y)$, we have
$$
\E[\varphi^\e_{t,x}(0,Y^a_0)]=||c_0||_{L_ 1(\R^2)}^{-1}  \theta (\nabla \phi_\e \ast K^{\theta,\lambda}_t\ast c_0)(x)
= ||c_0||_{L_ 1(\R^2)}^{-1} \nabla (\phi_\e\ast b^{c_0,\theta,\lambda}_t)(x), 
$$
recall~\eqref{gkb}.
Since moreover $\nu^1_0=\frac{1}{(N-1)M}\sum_{a\in \Uu_{N,M}^0}\delta_{Y^a_0}$  with $\# (\Uu_{N,M}^0)= L_{N,M}$, we find
\begin{align*}
\Delta^2_t(x)=&\frac1{(N-1)M}\sum_{a \in \Uu_{N,M}^0} \varphi^\e_{t,x}(0,Y^a_0)- \frac{L_{N,M}}{(N-1)M||c_0||_{L^1(\R^2)}} 
\nabla (\phi_\e\ast b^{c_0,\theta,\lambda}_t)(x)\\
=& \frac1{(N-1)M}\sum_{a \in \Uu_{N,M}^0} Z^a_t(x),
\end{align*}
where $Z^a_t(x)=\varphi^\e_{t,x}(0,Y^a_0)-\E[\varphi^\e_{t,x}(0,Y^a_0)]$.
Since $\E[Z^a_t(x)]=0$ for all $x \in \R^2$, it also holds true that $\E[\nabla Z^a_t(x)]=0$. 
We now fix $r>2$.
By the discrete 
Burkholder-Davies-Gundy inequality and since the family $(\nabla Z^a_t(x))_{a \in \Uu_{N,M}^0}$ is i.i.d.,
\begin{align*}
\E[|\nabla \Delta^2_t(x)|^r]\leq & \frac{A_r (L_{N,M})^{\frac r 2}}{(NM)^r} \E[|\nabla Z^{(1,0)}_t(x)|^r].
\end{align*}
But for any random variable $Y$ with mean $m$, $\E[|Y-m|^r]\leq A_r (\E[|Y|^r]+m^r)\leq A_r \E[|Y|^r]$.
Since moreover $L_{N,M}\leq A NM$,
\begin{align*}
\E[|\nabla \Delta^2_t(x)|^r]\leq & \frac{A_{r}}{(NM)^{\frac r2}} \int_{\R^2} |\nabla_x \varphi^\e_{t,x}(0,y)|^rc_0(y)\dd y.
\end{align*}
Integrating in $x$ and using~\eqref{estp0} and that $c_0 \in L^1(\R^2)$, we find
$$
\E[||\nabla \Delta^2_t||_{L^r(\R^2)}^r]\leq  \frac{A_r}{(NM)^{\frac r2} (t+\e)^{2r-1}}\leq  
\frac{A_r}{(NM)^{\frac r2} \e^{2r-1}}.
$$
One finds similarly
$$
\E[||\Delta^2_t||_{L^r(\R^2)}^r]\leq  \frac{A_r}{(NM)^{\frac r2} \e^{\frac{3r}2-1}}.
$$
As in Step~3, we get
$$
\E[||\Delta^2_t||_{L^\infty(\R^2)}^2]\leq A_r \E\Big[||\nabla \Delta^2_t||_{L^r(\R^2)}^r\Big]^{\frac{4}{r^2}}
\E\Big[||\Delta^2_t||_{L^r(\R^2)}^r\Big]^{\frac{2(r-2)}{r^2}}.
$$
This gives
\begin{equation*}
\E[||\Delta^2_t||_{L^\infty(\R^2)}^2]
\leq \frac{A_r}{NM \e^3}\leq \frac{A_\sigma}{NM \e^{3+\sigma}}.
\end{equation*}
\vskip-0.8cm
\end{proof}
\vskip0.6cm

\section{Main estimates and tightness of 2SKS}
\label{meat}

This section is devoted to some uniform in $N,M,\e$ estimates and to tightness.
More precisely, \eqref{eq:Prop21-1} shows that, despite the attraction,
cells do not spend too much time close to each other, while \eqref{eq:Prop21-3}
allows us to control the drift of the cells, which will imply tightness. Moreover,
\eqref{eq:Prop21-5} shows that the target drift in 1SKS, see~\eqref{PS}, is well-behaved.
Finally, \eqref{eq:Prop21-2} is an intermediate step to get~\eqref{eq:Prop21-5}, see~\eqref{tri1p}.
Recall that $\Aa_\sigma$ is defined in~\eqref{condie1}.

\begin{proposition}\label{prop:mc}
Let $\gamma \in (\frac 32,2)$ and $\alpha>0$. Assume that $\chi>0$ and 
$\theta>0$ are such that~\eqref{condistar} holds true.
For all $\sigma\in (0,\frac12]$, all $T>0$, there exist a constant $A_{\sigma,T}$  
(also depending on $\chi,\lambda,\alpha,\gamma,\theta,c_0$) such that for all $(N,M,\e)\in \Aa_\sigma$,
\begin{gather}
\E \Big[\int_0^T \int_{\R^2} \frac{1}{|X_s^{1,N,M,\e}-X_s^{2,N,M,\e}-\sqrt\e y|^{2(\gamma-1)}} \phi(y)\dd y\dd s\Big ]\leq A_{\sigma,T},
\label{eq:Prop21-1}\\
\E \Big[\int_0^T  \int_0^s \int_{\R^2}\frac{1}{(s-u+ |X_s^{1,N,M,\e}-X_u^{2,N,M,\e}-\sqrt\e y|^{2})^\gamma} \phi(y) \dd y 
\dd u \dd s \Big ]\leq A_{\sigma,T}, \label{eq:Prop21-2} \\
\E \Big[\int_0^T \Big|(\nu^{1,N,M,\e}_s \ast \nabla \phi_{\e})(X^{1,N,M,\e}_s)\Big|^{2(\gamma-1)}\dd s
\Big]\leq A_{\sigma,T}, \label{eq:Prop21-3}\\
\E \Big[\int_0^T  \int_0^s \int_{\R^2}
\Big|\nabla K^{\theta,\lambda}_{s-u}(X_s^{1,N,M,\e}-X_u^{2,N,M,\e}-\sqrt\e y)\Big|^{\frac{2\gamma}3} \phi(y)\dd y
\dd u \dd s \Big ]\leq A_{\sigma,T}. \label{eq:Prop21-5}
\end{gather}
\end{proposition}

In the whole section, we remove the superscripts $N,M$ and $N,M,\e$ for simplicity.
We also set
$$
R^{i,j}_{t,s}=X^i_t-X^i_s \quad \text{and}  \quad D^i_t=(\nu^{i}_s \ast \nabla \phi_{\e})(X^{i}_s).
$$
We start the proof of Proposition~\ref{prop:mc} with the following lemma where we bound the drift of the cell particles in the 2SKS system by a quantity which does not depend anymore on chemoattractant particles, but which rather involves the trajectories of the cells.

\begin{lemma}\label{nest}
For any $\sigma \in (0,\frac 12]$, any $\alpha>0$, any $\gamma \in (\frac 32,2)$, any $T>0$, any $\eta>0$,
there exists a constant $A_{\sigma,\eta,T}$ (also depending on $\lambda,\theta,\gamma,\alpha,c_0$) such that 
if $(N,M,\e)\in \Aa_\sigma$, for all $u\in [0,T]$,
\begin{align*}
\E\Big[\int_0^u |D^1_t|^{2\gamma-2} \dd t \Big] \leq &
\Big((1+\eta)C_4(\theta,\alpha,\gamma)\Big)^{2\gamma-2}
\E\Big[\int_0^u \int_0^t \int_{\R^2} \frac{\phi(y) \dd y\dd s\dd t}{(t-s+\alpha|R^{1,2}_{t,s}-\sqrt \e y|^2)^{\gamma}}\Big]
+ A_{\sigma,\eta,T},
\end{align*}
where $C_4(\theta,\alpha,\gamma)=C_3(\theta,\alpha)\kappa(\frac{1}{2}, \gamma-1\big)$.
\end{lemma}

\begin{proof} We set $\beta=2\gamma-2 \in (1,2)$.
Notice that
$D^1_t= \nabla (\phi_\e\ast b^{c_0,\theta,\lambda}_t)(X^1_t)+ V_t^1+ U^1_t$, where
$$
V_t^1=\frac1{N-1}\sum_{j\in I_N,j\neq 1}\int_0^t  (\nabla \phi_{\e}\ast K^{\theta,\lambda}_{t-s})(X^1_t-X^{j}_s)\dd s
$$
and with $\sup_{[0,T]} \E[|U^1_t|^2] \leq A_{\sigma,T}$ by Proposition~\ref{prel}, because $(N,M,\e)\in \Aa_\sigma$, see~\eqref{condie1}, which also implies that $NM\e^{\frac12} \geq NM\e^{3+\sigma}\geq 1$.
For $\eta>0$, there exists $A_{\eta}>0$ (also depending on $\beta$) 
such that $(a_1+a_2+a_3)^\beta\leq (1+\eta)^\beta a_1^\beta + A_{\eta}(a_2^\beta+a_3^\beta)$ for all 
$a_1,a_2,a_3\geq 0$. Thus for all $u\in [0,T]$, using that $\sup_{[0,T]} \E[|U^1_t|^\beta] \leq A_{\sigma,T}$,
\begin{align}\label{pf}
\E\Big[\int_0^u |D^1_t|^\beta\dd t\Big]\leq& (1+\eta)^\beta\E\Big[\int_0^u |V^1_t|^\beta\dd t\Big]
+ A_{\eta} \int_0^T ||\nabla (\phi_{\e}\ast  b^{c_0,\theta,\lambda}_t) ||^\beta_{L^\infty(\R^2)} \dd t 
+ A_{\eta,\sigma,T}.
\end{align}
By \eqref{rmu} and since $\beta \in (1,2)$, there exists $A_{T}$
(also depending on $c_0,\theta,\gamma$) such that
\begin{equation}\label{aa1}
\int_0^T ||\nabla (\phi_{\e}\ast b^{c_0,\theta,\lambda}_t) ||^\beta_{L^\infty(\R^2)} \dd t \leq A_{T}.
\end{equation}
Since $\beta>1$, we have
\begin{align*}
\E[|V^1_t|^\beta] \leq& \frac1{N-1}\sum_{j\in I_N,j\neq 1}\E\Big[\Big(\int_0^t  
|(\nabla \phi_{\e}\ast K^{\theta,\lambda}_{t-s})(X^1_t-X^{j}_s)|\dd s\Big)^\beta \Big]\\
=&\E\Big[\Big(\int_0^t  
|(\nabla \phi_{\e}\ast K^{\theta,\lambda}_{t-s})(X^1_t-X^{2}_s)|\dd s\Big)^\beta \Big]
\end{align*}
by exchangeability. Using next~\eqref{tri1} and integrating in $t\in [0,u]$, we get 
\begin{align*}
\int_0^u \E[|V^1_t|^\beta] \dd t\leq& \Big(C_3(\theta,\alpha)\Big)^\beta
\E\Big[\int_0^u \Big(\int_{\R^2}\int_0^t (t-s+\alpha |X^1_t-X^2_s-\sqrt\e y|^2)^{-\frac{3}{2}} \dd s\phi(y)\dd y\Big)^\beta \dd t\Big]\\
\leq& \Big(C_3(\theta,\alpha)\Big)^\beta
\E\Big[\int_0^u  \int_{\R^2}\Big(\int_0^t (t-s+\alpha |X^1_t-X^2_s-\sqrt\e y|^2)^{-\frac{3}{2}} \dd s\Big)^\beta \phi(y)\dd y\dd t\Big]
\end{align*}
by Hölder's inequality, since $\beta>1$ and since $\phi$ is a probability density.
Using finally Lemma~\ref{lemma:FI} with $a=\frac 12$ and $b=\gamma-1>a$, we get
\begin{equation*}
\E\Big[\int_0^u |V^1_t|^\beta \dd t \Big] \leq
\Big(C_3(\theta,\alpha)\kappa\Big(\frac12,\gamma-1\Big)\Big)^\beta
\E\Big[\int_0^u \int_{\R^2}\int_0^t  (t-s+\alpha |X^1_t-X^2_s-\sqrt\e y|^2)^{-\gamma} 
\dd s \phi(y)\dd y\dd t\Big].
\end{equation*}
Inserting~\eqref{aa1} and this last inequality
into~\eqref{pf} completes the proof.
\end{proof}

The rest of the section closely follows~\cite{FT1}. We start with a specific Itô formula.

\begin{lemma}\label{lemma:grosIto}
Let $F: \R_+ \times \R^2 \to \R$ be of class $C^{2}_b(\R_+ \times \R^2)$. For all $u\geq 0$,
\begin{align*}
\E \Big[ \int_0^u F(u-s, R^{1,2}_{u,s}) \dd s \Big ] =& \E \Big[ \int_0^u F(0, R^{1,2}_{s,s})
\dd s \Big ] + \E \Big[\int_0^u \int_0^t (\partial_t F + \Delta F) (t-s,R^{1,2}_{t,s})  \dd s \dd t \Big]\\
&+ \chi \E \Big[\int_0^u \Big(\int_0^t \nabla F (t-s, R_{t,s}^{1,2})  \dd s \Big) \cdot D^{1}_t \dd t \Big].
\end{align*}
\end{lemma}

\begin{proof}
It suffices to use the arguments of~\cite[Lemma~10]{FT1}, replacing the drift of $(X^1_t)_{t\geq 0}$ therein,
namely $\nabla b^{c_0,\theta,\lambda}(X^1_t) + \frac1{N-1}\sum_{j=2}^N D^{1,j}_t$ by the drift 
of $(X^1_t)_{t\geq 0}$ here, namely $D^{1}_t$.
\end{proof}

Once we have bounded in Lemma~\ref{nest} the drift of the 2SKS cell particles by a quantity involving trajectorial interactions (as in the drift of the 1SKS system), we are in a position to adapt~\cite[Proposition~11]{FT1} to our situation: 
We can bound these trajectorial interactions by a term involving only instantaneous interactions.

\begin{lemma}\label{lemma:LL}
Fix $\alpha>0$ and $\gamma\in (\frac 32,2)$. Assume that $\chi>0$ and $\theta>0$ are such that
$C_1(\alpha,\gamma)>0$ and $C_1(\alpha,\gamma)>\chi C_2(\theta,\alpha,\gamma)$.
For all $T>0$, all $\sigma\in (0,\frac 12]$ and all $\zeta>0$, there exists a constant $A_{\sigma,\zeta,T}$
(also depending on $\chi,\theta,c_0,\gamma,\alpha$) such that if $(N,M,\e)\in \Aa_\sigma$, then for
all $u \in [0,T]$,
\begin{align*}
&\E\Big[ \int_0^u \int_0^t \int_{\R^2}\frac{1}{(t-s + \alpha |R^{1,2}_{t,s}-\sqrt\e y|^2)^{\gamma}} \phi(y)\dd y\dd s\dd t \Big]\\
\leq& \frac{(1+\zeta)\alpha^{1-\gamma}}{C_1(\alpha,\gamma) -\chi C_2(\theta,\alpha,\gamma)} 
\E\Big[ \int_0^u \int_{\R^2} \frac{1}{|R_{s,s}^{1,2}-\sqrt\e y|^{2(\gamma-1)}}\phi(y)\dd y \dd s\Big] + A_{\sigma,\zeta,T}.
\end{align*}
\end{lemma}

\begin{proof}
We set
$$
S^{1,2}_t=\int_0^t \int_{\R^2} (t-s + \alpha |R^{1,2}_{t,s}-\sqrt\e y|^2)^{-\gamma} \phi(y) \dd y  \dd s
$$
and we define $F:\R_+\times\R^2\to \R_-$ by $F(t, x)= - (t+\alpha|x|^2)^{1-\gamma}$. We have as in~\cite{FT1}
\begin{gather*}
\nabla F(t,x)= 2\alpha (\gamma-1) (t+\alpha|x|^2)^{-\gamma} x, \quad
\text{whence} \quad |\nabla F(t,x)| \leq 2\sqrt\alpha (\gamma-1)(t+\alpha|x|^2)^{\frac{1}{2}-\gamma}, \\
(\partial_t F + \Delta F)(t,x) \geq C_1(\alpha,\gamma)(t+\alpha|x|^2)^{-\gamma}.
\end{gather*}
We next set $F_\e(t,x)=(F(t,\cdot)\ast \phi_\e)(x)$, which satisfies
\begin{gather*}
F_\e(0,x)=-\int_{\R^2} F(0,x-y)\phi_\e(y) \dd y =-\alpha^{1-\gamma}\int_{\R^2} |x-\sqrt{\e}y |^{2(1-\gamma)}\phi(y)\dd y, \\
|\nabla F_\e(t,x)|\leq \int_{\R^2} |\nabla F(t,x-y)|\phi_\e(y) \dd y\leq 
2\sqrt\alpha (\gamma-1)\int_{\R^2}(t+\alpha|x-\sqrt\e y|^2)^{\frac{1}{2}-\gamma} \phi(y)\dd y,\\
(\partial_t F_\e + \Delta F_\e)(t,x) = \int_{\R^2}(\partial_t F + \Delta F)(t,x-y) \phi_\e(y)\dd y \geq C_1(\alpha,\gamma)\int_{\R^2}(t+\alpha|x-\sqrt\e y|^2)^{-\gamma} \phi(y)\dd y.
\end{gather*}
Applying  Lemma~\ref{lemma:grosIto} to $F_\e$, we get
\begin{equation}\label{eq:abeg}
I^{1}_u= I^{2}_u+I^{3}_u + I^{4}_u,
\end{equation}
where, using the above estimates concerning $F_\e$,
\begin{align}
I^{1}_u:=&\E \Big[\int_0^u F_\e(u-s, R^{1,2}_{u,s}) \dd s\Big]\leq 0, \label{eq:a}\\
I^{2}_u :=&\E \Big[\int_0^u F_\e(0, R^{1,2}_{s,s}) \dd s\Big]= -\alpha^{1-\gamma}\E \Big[\int_0^u \int_{\R^2}|R^{1,2}_{s,s}-\sqrt{\e}y|^{2(1-\gamma)} 
\phi(y) \dd y \dd s\Big],\label{eq:b}\\
I^{3}_u:=&  \E \Big[\int_0^u \int_0^t (\partial_t F_\e + \Delta F_\e) (t-s, R_{t,s}^{1,2})  \dd s \dd t \Big]
\geq C_1(\alpha,\gamma)\E\Big[\int_0^u  S^{1,2}_t \dd t \Big], \label{eq:c}\\ 
I^{4}_u:=&  \chi \E \Big[\int_0^u \Big(\int_0^t  \nabla F_\e(t-s, R_{t,s}^{1,2})\dd s\Big)\cdot D^{1}_t \dd t\Big]. 
\notag
\end{align}
We write $I^4_u \geq - \chi \E[\int_0^u  |D^{1}_t| T_t \dd t]$, where, by the above estimate on $\nabla F_\e$, 
$$T_t=\int_0^t  |\nabla F_\e(t-s, R_{t,s}^{1,2})| \dd t \leq 
2\sqrt{\alpha}(\gamma-1)\int_0^t \int_{\R^2}(t-s+\alpha| R^{1,2}_{t,s}-\sqrt \e y|^2)^{\frac12-\gamma} \phi(y)\dd y\dd s.
$$
By Lemma~\ref{lemma:FI} with $a=\gamma-\frac 32 $ and $b=\gamma-1$,
setting 
\begin{equation}
C_5(\alpha,\gamma)=2\sqrt{\alpha}(\gamma-1)\kappa\Big(\gamma-\frac32,\gamma-1\Big),\label{estTu}
\end{equation}
we have 
\begin{align*}
T_t\leq& C_5(\alpha,\gamma) \int_{\R^2}\Big(\int_0^t (t-s+\alpha| R^{1,2}_{t,s}-\sqrt \e y|^2)^{-\gamma}\dd s\Big)^{\frac{\gamma-\frac 32}{\gamma-1}} \phi(y)\dd y
\\
\leq & C_5(\alpha,\gamma) \Big(\int_{\R^2}\int_0^t (t-s+\alpha| R^{1,2}_{t,s}-\sqrt \e y|^2)^{-\gamma}\dd s\phi(y)\dd y\Big)^{\frac{\gamma-\frac 32}{\gamma-1}} 
\end{align*}
by Hölder's inequality with respect to $\phi(y)\dd y$. Thus $T_t\leq C_5(\alpha,\gamma) (S^{1,2}_t)^{\frac{\gamma-\frac 32}{\gamma-1}}$, so that
\begin{align*}
I^4_u \geq& - \chi C_5(\alpha,\gamma) \E \Big[\int_0^u  |D^{1}_t| (S^{1,2}_t)^ {\frac{\gamma-\frac 32}{\gamma-1}}\dd t\Big]\\
\geq& - \chi C_5(\alpha,\gamma) \E \Big[\int_0^u  |D^{1}_t|^{2\gamma-2}\dd t\Big]^{\frac1{2\gamma-2}}
\E\Big[\int_0^u S^{1,2}_t\dd t\Big]^{\frac{\gamma-\frac 32}{\gamma-1}}
\end{align*}
by Hölder's inequality with  $p=2\gamma-2$ and $p'=\frac{2\gamma-2}{2\gamma-3}$.
By Lemma~\ref{nest}, for $\eta>0$ to be chosen later, for all $u \in [0,T]$,
\begin{align*}
I^4_u \geq& - \chi C_5(\alpha,\gamma)\Big([(1+\eta)C_4(\theta,\alpha,\gamma)]^{2\gamma-2}
\E\Big[\int_0^u S^{1,2}_t\dd t\Big]+A_{\sigma,\eta,T}\Big) ^{\frac1{2\gamma-2}}
\E\Big[\int_0^u S^{1,2}_t\dd t\Big]^{\frac{2\gamma-3}{2\gamma-2}}.
\end{align*}
Since $2\gamma-2>1$ and $\eta>0$, there exists $A_{\eta}>0$ such that for all $a_1,a_2,a_3 \geq0$
\begin{align*}
(a_1^{2\gamma-2}a_2+a_3)^{\frac1{2\gamma-2}}a_2^{\frac{2\gamma-3}{2\gamma-2}} \leq & (a_1a_2^{\frac1{2\gamma-2}}
+a_3^{\frac1{2\gamma-2}}) a_2^{\frac{2\gamma-3}{2\gamma-2}}
= a_1a_2 +a_3^{\frac1{2\gamma-2}}a_2^{\frac{2\gamma-3}{2\gamma-2}}
\leq (a_1+\eta) a_2 + A_{\eta} a_3
\end{align*}
by Young's inequality. Thus for all $u \in [0,T]$, with $A_{\sigma,\eta,T}$ also depending on $\chi$,
\begin{align}
I^4_u\geq& - \chi C_5(\alpha,\gamma)[(1+\eta)C_4(\theta,\alpha,\gamma)+\eta]
\E \Big[\int_0^u   S^{1,2}_t\dd t\Big] - A_{\sigma,\eta,T}.
\label{eq:d}
\end{align}

Plugging~\eqref{eq:a}-\eqref{eq:b}-\eqref{eq:c}-\eqref{eq:d} into~\eqref{eq:abeg}, we get
\begin{align*}
&\Big(C_1(\alpha,\gamma)\!-\!\chi C_5(\alpha,\gamma)[(1\!+\!\eta)C_4(\theta,\alpha,\gamma)\!+\!\eta]\Big)
\E \Big[\int_0^u \!  S^{1,2}_t\dd t\Big]\\
\leq& \alpha^{1-\gamma}  \E \Big[\int_0^u\!\int_{\R^2} |R^{1,2}_{s,s}-\sqrt\e y|^{2(1-\gamma)} \phi(y)\dd y\dd s\Big]\!+\!A_{\sigma,\eta,T}.
\end{align*}
Recalling   \eqref{estTu},  Lemma~\ref{nest}, \eqref{tri1p} and~\eqref{eq:C2}, one sees that 
$C_5(\alpha,\gamma)C_4(\theta,\alpha,\gamma)=C_2(\theta,\alpha,\gamma)$.
Recall also that $C_1(\alpha,\gamma)>\chi C_2(\theta,\alpha,\gamma)$ by assumption. 
Thus for $\zeta>0$ fixed, we can find $\eta>0$ small enough such that
$$
\Big(C_1(\alpha,\gamma)\!-\!\chi C_5(\alpha,\gamma)[(1\!+\!\eta)C_4(\theta,\alpha,\gamma)\!+\!\eta]\Big)
\geq \frac{C_1(\alpha,\gamma)-\chi C_2(\theta,\alpha,\gamma)}{1+\zeta}>0.
$$
As a conclusion, for all $u\in [0,T]$,
$$
\E \Big[\int_0^u \!  S^{1,2}_t\dd t\Big] \leq 
\frac{ \alpha^{1-\gamma} (1+\zeta)}{C_1(\alpha,\gamma)-\chi C_2(\theta,\alpha,\gamma)} 
\E \Big[\int_0^u\! \int_{\R^2} \frac{\phi(y)\dd y\dd s}{|R^{1,2}_{s,s}-\sqrt\e y|^{2(\gamma-1)}}\Big]
+ \frac{(1+\zeta)A_{\sigma,\zeta,T}}{C_1(\alpha,\gamma)-\chi C_2(\theta,\alpha,\gamma)},
$$
which was our goal.
\end{proof}

We can finally handle the

\begin{proof}[Proof of Proposition~\ref{prop:mc}]
We follow the arguments of~\cite[Proof of Proposition~8]{FT1}.

\vip

Define $\varphi(r):=(1+ r^{2-\gamma})^{-1}r^{2-\gamma}$ on $\R_+$ and set 
$\psi(x,y)=\varphi(|x-y|^2)$ for $x,y \in \R^2$. 
As in~\cite{FT1}, for all $\eta>0$, there exists $A_\eta>0$ (also depending on $\gamma$) such that for all $x,y \in \R^2$,
\begin{gather*}
\nabla_x \psi (x,y)= (4-2\gamma) \frac{|x-y|^{2-2\gamma}}{(1+|x-y|^{4-2\gamma})^2}(x-y), \quad \text{whence} \quad  |\nabla_x\psi(x,y)|\leq (4-2\gamma)|x-y|^{3-2\gamma},\\ 
\Delta_x \psi (x,y) \geq ((4-2\gamma)^2-\eta) |x-y|^{2-2\gamma}-A_\eta.
\end{gather*}
We define, using the symmetry of $\phi_\e$, see Assumption~\ref{mol},
$$
\psi_\e(x,y)=\int_{\R^2}\psi (x,z) \phi_\e(y-z)\dd z = \psi_\e(y,x).
$$
Notice that, using again the symmetry of $\phi$,
\begin{gather*}
|\nabla_x \psi_\e(x,y)|\leq \int_{\R^2} |\nabla_x \psi(x,z)|\phi_\e(y-z)\dd z
\leq (4-2\gamma) \int_{\R^2} |x-y-\sqrt{\e} z|^{3-2\gamma} \phi(z) \dd z,\\
\Delta_x \psi_\e(x,y)=\int_{\R^2} \Delta_x\psi (x,z) \phi_\e(y-z)\dd z  \geq 
((4-2\gamma)^2-\eta) \int_{\R^2} |x-y-\sqrt\e z|^{2-2\gamma}\phi(z)\dd z-A_\eta.
\end{gather*}
By the It\^o formula starting from the first line of~\eqref{troc}, by symmetry and exchangeability, recalling that $D^1_t= (\nu^1_t\ast \nabla \phi_\e)(X^1_t)$,
\begin{align*}
\E[\psi_\e (X^1_T,X^2_T)]=\E[\psi_\e (X^1_0,X^2_0)]+ 2\E\Big[\int_0^T \Delta_x \psi_\e (X^1_s,X^2_s) \dd s\Big]
+ 2\chi \E\Big[\int_0^T  \nabla_x \psi_\e (X^1_s,X^2_s) \cdot D^{1}_s\dd s  \Big].
\end{align*}
Using that $\psi_\e$ takes values in $[0,1]$ and the above lowerbound of $\Delta \psi_\e$, we conclude that
$$
1\geq 2 ((4-2\gamma)^2-\eta)\E\Big[\int_0^T\int_{\R^2} |R^{1,2}_{s,s}-\sqrt\e z|^{2\gamma-2} \phi(z)\dd z\dd s \Big] - 2A_\eta T+2 \chi H_T,
$$
i.e. that 
\begin{equation}\label{tbpi}
((4-2\gamma)^2-\eta)\E\Big[\int_0^T \int_{\R^2}|R^{1,2}_{s,s}-\sqrt\e z|^{2\gamma-2} \phi(z)\dd z \dd s \Big] +\chi H_T \leq  \frac 12+A_\eta T,
\end{equation}
where
$$
H_T=\E\Big[\int_0^T  \nabla_x \psi_\e (X^1_s,X^2_s) \cdot D^{1}_s\dd s  \Big]
\geq -(4-2\gamma) \E \Big[\int_0^T \int_{\R^2}  |R^{1,2}_{s,s}-\sqrt\e z|^{3-2\gamma} \phi(z)\dd z |D^1_s| \dd s \Big].
$$
By the H\"older inequality (for $\E[\int_0^T \int_{\R^2} \;\cdot\; \phi(z)\dd z \dd s]$) with  $p=\frac{2\gamma-2}{2\gamma-3}$ and $p'=2\gamma-2$,
\begin{align*}
H_T \geq& -(4-2\gamma) \E \Big[\int_0^T \int_{\R^2} |R^{1,2}_{s,s}-\sqrt\e z|^{2-2\gamma} \phi(z)\dd z\dd s \Big]^{\frac{2\gamma-3}{2\gamma-2}}
 \E \Big[\int_0^T  |D^{1}_{s}|^{2\gamma-2} \dd s \Big]^{\frac{1}{2\gamma-2}}.
\end{align*}
By Lemmas~\ref{nest} and~\ref{lemma:LL} (which is licit since 
$C_1(\alpha,\gamma)>\chi C_2(\theta,\alpha,\gamma)$ by assumption), for $\eta>0$ to be chosen later,
\begin{gather}
\E \Big[\int_0^T  |D^{1}_{s}|^{2\gamma-2} \dd s \Big]\leq  (1+\eta) C_6(\chi,\theta,\alpha,\gamma)\E\Big[
\int_0^T \int_{\R^2}  |R^{1,2}_{s,s}-\sqrt \e z|^{2-2\gamma} \phi(z)\dd z\dd s\Big]
+A_{\sigma,\eta,T},\notag\\
\text{where} \quad C_6(\chi,\theta,\alpha,\gamma)= \frac{(C_4(\theta,\alpha,\gamma))^{2\gamma-2}\alpha^{1-\gamma}}
{C_1(\alpha,\gamma)-\chi C_2(\theta,\alpha,\gamma)}. \label{C6}
\end{gather}
Thus
\begin{align*}
&H_T\geq -(4-2\gamma) \E \Big[\int_0^T \int_{\R^2} |R^{1,2}_{s,s}-\sqrt\e z|^{2-2\gamma} \phi(z)\dd z\dd s \Big]^{\frac{2\gamma-3}{2\gamma-2}}\\
&\hskip3cm\times\Big((1+\eta) C_6(\chi,\theta,\alpha,\gamma)\E\Big[
\int_0^T \int_{\R^2} |R^{1,2}_{s,s}-\sqrt\e z|^{2-2\gamma} \phi(z)\dd z\dd s\Big]
+A_{\sigma,\eta,T}\Big)^{\frac1{2\gamma-2}}.
\end{align*}
Since $2\gamma-2>1$ and $\eta>0$, there exists $A_{\eta}>0$ such that for all $a_1,a_2,a_3\geq 0$,
$$
a_1^{\frac{2\gamma-3}{2\gamma-2}}(a_2a_1+a_3)^{\frac1{2\gamma-2}}\leq a_2^{\frac1{2\gamma-2}}a_1
+a_1^{\frac{2\gamma-3}{2\gamma-2}}a_3^{\frac1{2\gamma-2}}\leq (a_2^{\frac1{2\gamma-2}}+\eta)a_1+A_{\eta}a_3
$$
by Young's inequality. Thus
$$
H_T\geq -(4-2\gamma) \Big([(1+\eta) C_6(\chi,\theta,\alpha,\gamma)]^{\frac 1{2\gamma-2}}+\eta\Big)
\E \Big[\int_0^T \int_{\R^2} |R^{1,2}_{s,s}-\sqrt\e z|^{2-2\gamma}\phi(z)\dd z \dd s \Big] - A_{\sigma,\eta,T}.
$$
Inserting this estimate in~\eqref{tbpi}, we find, for all $t\in [0,T]$,
\begin{align}
&\Big((4-2\gamma)^2 - \eta-\chi (4-2\gamma) 
\Big([(1+\eta) C_6(\chi,\theta,\alpha,\gamma)]^{\frac 1{2\gamma-2}}+\eta\Big) \Big) \notag\\ &\hskip5cm \times
\E \Big[\int_0^T \int_{\R^2}  |R^{1,2}_{s,s}-\sqrt{\e}y|^{2-2\gamma} \phi(y) \dd y\dd s \Big] \leq A_{\sigma,\eta,T}. \label{tlo}
\end{align}
Recalling~\eqref{C6}, Lemma~\ref{nest} and \eqref{tri1p}, we have
$$
(4-2\gamma)^2 -\chi(4-2\gamma) [C_6(\chi,\theta,\alpha,\gamma)]^{\frac 1{2\gamma-2}}=
(4-2\gamma)\Big[(4-2\gamma)-\frac{\chi\sqrt{\theta} C_0\big(\frac{4\alpha}{\theta}\big) 
\kappa\big(\frac{1}{2}, \gamma-1\big)}{4\pi\sqrt{\alpha}[C_1(\alpha,\gamma) 
-\chi C_2(\theta,\alpha,\gamma)]^{\frac 1{2(\gamma-1)}}}\Big]
>0
$$
by assumption, see~\eqref{condistar}. Thus if choosing $\eta>0$ small enough,
$$
\Big((4-2\gamma)^2 - \eta-(4-2\gamma) \Big([(1+\eta) C_6(\chi,\theta,\alpha,\gamma)]^{\frac 1{2\gamma-2}}+\eta\Big) 
\Big)>0.
$$
Thus~\eqref{eq:Prop21-1} follows from~\eqref{tlo}.
\vip
By Lemma~\ref{lemma:LL}, \eqref{eq:Prop21-1} implies that 
$$
E\Big[\int_0^T \int_0^t \int_{\R^2}(t-s+\alpha|R^{1,2}_{t,s}-\sqrt\e z|^2)^{-\gamma} \phi(z)\dd z\dd s\dd t\Big]\leq A_{\chi,\sigma,\zeta,T}.
$$
From this last estimate, we deduce~\eqref{eq:Prop21-2} (directly), \eqref{eq:Prop21-3} 
(through Lemma~\ref{nest}, recall that $D^1_t=(\nu_t^1\ast\nabla \phi_\e)(X^1_t)$) 
and~\eqref{eq:Prop21-5} (through~\eqref{tri1p}).
\end{proof}

Finally, we follow the main ideas of~\cite[Lemma~11]{fournier-jourdain} to prove tightness.

\begin{proof}[Proof of Theorem~\ref{main}-(i)] By Remark~\ref{chistar} and since $\chi<\chi^*_\theta$, 
we can find $\alpha>0$ and  $\gamma \in (\frac32,2)$ such that \eqref{condistar} holds true.
By Proposition~\ref{prop:mc}, for all $T>0$, there exists $A_{T}>0$ such that for all $(N,M,\e)\in \Aa_\sigma$,
\begin{equation}\label{tic}
\E\Big[\int_0^T  |D^{1,N,M,\e}_s|^{2\gamma-2}\dd s\Big] \leq A_T,\quad\text{where}\quad D^{1,N,M,\e}_s=
(\nu^{1,N,M,\e}_s\ast \nabla \phi_\e)(X^{1,N,M,\e}_s).
\end{equation}
Recall from~\eqref{troc} that $X^{1,N,M,\e}_t=X^1_0+\sqrt 2 B^1_t+ \chi H^{1,N,M,\e}_t$,
where $H^{1,N,M,\e}_t=\int_0^t D^{1,N,M,\e}_s \dd s$.
To show the tightness of $((X^{1,N,M,\e}_t)_{t\geq 0} : (N,M,\e)\in \Aa_\sigma)$ in $C(\R_+,\R^2)$, 
it suffices to show that for each $T>0$, the family $((H^{1,N,M,\e}_t)_{t\in [0,T]}: (N,M,\e)\in \Aa_\sigma)$ 
is tight in $C([0,T],\R^2)$.
\vip
By Hölder's inequality with $p=2\gamma-2$ and $p'=\frac{2\gamma-2}{2\gamma-3}$, 
we see that for all $0\leq s \leq t \leq T$, 
$$
|H^{1,N,M,\e}_t-H^{1,N,M,\e}_s|\leq 
(t-s)^{\frac{2\gamma-3}{2\gamma-2}} Z^{1,N,M,\e}_T,
\quad \text{where} \quad Z^{1,N,M,\e}_T=\Big(\int_0^T  |D^{1,N,M,\e}_u|^{2\gamma-2}\dd u\Big)^{\frac{1}{2\gamma-2}}.
$$
But $\E[Z^{1,N,M,\e}_T]\leq 1+A_T$ by~\eqref{tic} and since $2\gamma-2>1$. Thus
for $\Kk_{T,L}$ the set of continuous functions $x:[0,T]\to \R^2$ such that $x(0)=0$ and 
$|x(t)-x(s)|\leq L|t-s|^{\frac{2\gamma-3}{2\gamma-2}}$ for all $0\leq s \leq t \leq T$,
$$
\sup_{(N,M,\e)\in \Aa_\sigma}\PP((H^{1,N,M,\e}_t)_{t\in [0,T]} \notin \Kk_{T,L}) \leq 
\sup_{(N,M,\e)\in \Aa_\sigma}\PP(Z^{1,N,M,\e}_T>L)\leq \frac{1+A_T}L \to 0
$$
as $L\to \infty$. Since $\Kk_{T,L}$ is compact in $C([0,T],\R^2)$ for each $L\geq 1$, the
conclusion follows.
\end{proof}

\section{From the two-species system to the one-species system}
\label{tstos}

Following the lines of~\cite[Theorem~5]{fournier-jourdain} and~\cite[Proposition~13]{FT1}, we handle the 

\begin{proof}[Proof of Theorem~\ref{main}-(ii)]
We fix $N\geq 2$, as well as $\theta>0$, $\lambda>0$ and $\chi \in (0,\chi^*_\theta)$. 
By Remark~\ref{chistar}, we can find $\alpha>0$ and $\gamma\in (\frac32,2)$ such that~\eqref{condistar} 
holds true. We also consider some sequences $(\e_n)_{n\geq 1}$ valued in $(0,1]$ and
$(M_n)_{n\geq 1}$ valued in $(0,\infty)$ satisfying~\eqref{condie2}
for some $\sigma>0$ fixed.
For all $n$ large enough, we have $N M_n^{1+\sigma}\e_n^{2+\sigma}\geq N M_n\e_n^{3+\sigma}\geq 1$,
so that $(N,M_n,\e_n)\in \Aa_\sigma$. For each $n\geq 1$, we consider 
the $(\rho_0,c_0,N,M_n,\e_n)$-2SKS system $(X^{k,n}_t)_{t\geq 0, k\in I_N}$.

\vip

{\it Step 1.} By (i), the family
$((X^{1,n}_t)_{t\geq 0}, n \geq 1)$ is tight in $C(\R_+,\R^2)$, so that by exchangeability, the family
$((X^{i,n}_t)_{t\geq 0, i\in I_N}, n \geq 1)$ is tight in $C(\R_+,(\R^2)^N)$. This of course implies that
$((X^{i,n}_t,B^i_t)_{t\geq 0, i\in I_N}, n \geq 1)$ is tight in $C(\R_+,(\R^2\times \R^2)^N)$ and we consider
a (not relabeled) subsequence that $(X^{i,n}_t,B^i_t)_{t\geq 0, i\in I_N}$ converges in law
in $C(\R_+,(\R^2\times \R^2)^N)$ as $n\to \infty$.
\vip
By the Skorokhod representation theorem, we can find, for each $n\geq 1$, 
a $(\rho_0,c_0,N,M_n,\e_n)$-2SKS system $(\tX^{i,n}_t)_{t\geq 0, i \in I_N}$ associated to some Brownian
motions $(\tB^{i,n}_t)_{t\geq 0, i \in I_N}$, in such a way that 
$(\tX^{i,n}_t,\tB^{i,n}_t)_{t\geq 0, i\in I_N}$ a.s. goes to some
$(\tX^{i}_t,\tB^{i}_t)_ {t\geq 0,i\in I_N}$, as $n\to \infty$, in $C(\R_+,(\R^2\times\R^2)^{N})$.
Of course, $(\tX^{i}_t)_{t\geq 0,i \in I_N}$ inherits the exchangeability property.
It remains to prove that  $(\tX^{i,N}_t)_{t\geq 0,i\in I_N}$ is a $(\rho_0,c_0,N)$-1SKS system,
recall Definition~\ref{def:PS}.
\vip
We let $\Ff_t=\sigma((\tX_s^{i},\tB^{i}_s)_{s\in [0,t], i\in I_N})$, to which $(\tX^{i}_t)_{t\geq 0, i \in I_N}$ 
is adapted. One proves as in~\cite[Theorem~5]{fournier-jourdain} that
$(\tB^{i}_t)_{t\geq 0, i\in I_N}$ is a $2N$-dimensional $(\Ff_t)_{t\geq 0}$-Brownian motion.

\vip

Let us show that $(\tX^{i}_t)_{t\geq 0, i \in I_N}$ satisfies~\eqref{PScond}. 
By exchangeability, it suffices to study $(i,j)=(1,2)$. 
By the Fatou lemma (and since $\nabla K^{\theta,\lambda}_r$ is continuous on $\R^2$ for all $r>0$), we see that for all $T>0$,
\begin{align*}
\E\Big[\int_0^T \int_0^s |\nabla K^{\theta,\lambda}_{s-u} (\tX^{1}_s-&\tX^{2}_u)|^{\frac{2\gamma}3}  \dd u \dd s\Big]
=\E\Big[\int_0^T \int_0^s \int_{\R^2}|\nabla K^{\theta,\lambda}_{s-u} (\tX^{1}_s-\tX^{2}_u)|^{\frac{2\gamma}3} \phi(y)\dd y \dd u \dd s\Big]\\
\leq & \liminf_N \E\Big[\int_0^T \int_0^s \int_{\R^2}|\nabla K^{\theta,\lambda}_{s-u} 
(\tX^{1,n}_s-\tX^{2,n}_u-\sqrt{\e_n} y)|^{\frac{2\gamma}3} \phi(y)\dd y \dd u \dd s\Big],
\end{align*}
which is finite by~\eqref{eq:Prop21-5}.
Since $\frac{2\gamma}3>1$, the conclusion follows.

\vip

We finally verify $(\tX^{k,N}_t)_{t\geq 0,k\in I_N}$  satisfies~\eqref{PS} with the Brownian motions $(\tB^{k,N}_t)_{t\geq 0,k\in I_N}$ , i.e. that for each $k\in I_N$, each $t\geq 0$,
$$
\tX^{k}_t=\tX^{k}_0 + \sqrt 2 \tB^{k}_t + \chi V^{k}_t+ \frac \chi {N-1} \sum_{i \in I_N,i\neq k} Z^{k,i}_t,
$$
where
$$
V^{k}_t=\int_0^t \nabla b^{c_0,\theta,\lambda}_s(\tX^{k}_s)\dd s
\quad \hbox{and}\quad Z^{k,i}_t=
\int_0^t \int_0^s \nabla K_{s-u}^{\theta,\lambda}(\tX^{k}_s-\tX^{i}_u) \dd u \dd s.
$$
We start from~\eqref{troc}, i.e.
\begin{equation}\label{tpttli}
\tX^{k,n}_t=\tX^{k}_0 + \sqrt 2 \tB^{k,n}_t 
+ \chi \int_0^t (\tilde \nu^{k,n}_s \ast \nabla \phi_{\e_{n}})(\tX^{k,n}_s)\dd s,
\end{equation}
where we write $\tilde \nu^{k,n}_s$ for the chemoattractant empirical concentration corresponding to the 
2SKS system $(\tX^{k,n}_t)_{t\geq 0, k \in I_N}$. By Proposition~\ref{prel} (and exchangeability),
it holds that
$$
\int_0^t(\tilde \nu^{k,n}_s \ast \nabla \phi_{\e_{n}})(\tX^{k,n}_s)\dd s= \int_0^t U^{k,n}_s \dd s+V^{k,n}_t+
\frac1{N-1} \sum_{i \in I_N,i\neq k} Z^{k,i,n}_t,
$$
where $U^{k,n}_t$ is as in Proposition~\ref{prel} and
$$
V^{k,n}_t=\int_0^t(\nabla \phi_{\e_{n}}\ast b^{c_0,\theta,\lambda}_s)(\tX^{k,n}_s)\dd s
\quad \text{and}\quad  Z^{k,i,n}_t=\int_0^t \int_0^s(\nabla \phi_{\e_{n}}\ast K^{\theta,\lambda}_{s-u})(\tX^{k,n}_s
-\tX^{i,n}_u)\dd u \dd s.
$$
We now take the limit $n\to \infty$ in probability in~\eqref{tpttli}.

\vip
By construction, 
$(\tX^{k,n}_t,\tX^{k,n}_0,\tB^{k,n}_t)$ a.s. tends to $(\tX^{k}_t,\tX^{k}_0,\tB^k_t)$.

\vip
By Proposition~\ref{prel}, 
$$
\E[\Big[\Big|\int_0^t U^{k,n}_s \dd s\Big|\Big] \leq T\sup_{t\in [0,T]}\E[|U^{k,n}_t|^2]^{\frac12}
\leq T\Big(\frac{A_{\sigma,T}}{(NM_n \e_n^{1/2})^{2}}
+\frac{A_{\sigma,T}}{NM_n \e_n^{3+\sigma}}+\frac{A_{\sigma,T}}{NM_n^{1+\sigma}\e_n^{2+\sigma}}\Big)^{\frac12}
$$
which tends to $0$ as $n\to \infty$ by~\eqref{condie2}. Indeed, $\lim_{n}M_n \e_n^{3+\sigma}=\infty$
and $M_n \e_n^{1/2}\geq M_n\e_n^{3+\sigma}$ and $M_n^{1+\sigma}\e_n^{2+\sigma}\geq M_n\e_n^{3+\sigma}$ for all
$n$ large enough such that $M_n\geq 1$. Thus $\int_0^t U^{k,n}_s \dd s\to 0$ in probability.

\vip

Next, $V^{k,n}_t$ a.s. tends to $V^{k}_t$ by dominated convergence, because:
\vip
\noindent $\bullet$ by~\eqref{rmu}, $||\nabla \phi_{\e_n}\ast b^{c_0,\theta,\lambda}_s||_{L^\infty(\R^2)}\leq A s^{-1/2}$, 
which is time-integrable;
\vip
\noindent $\bullet$ for all $s>0$,  $(\nabla \phi_{\e_{n}}\ast b^{c_0,\theta,\lambda}_{s})(\tX^{i,n}_s)$ a.s. tends to 
$\nabla b^{c_0,\theta,\lambda}_{s}(\tX^{i}_s)$ by~\eqref{rtt1}.

\vip
Finally, 
\begin{align*}
Z_t^{k,i,n}=&\int_0^t \int_0^s \int_{\R^2}\nabla K^{\theta,\lambda}_{s-u}(\tX^{k,n}_s-\tX^{i,n}_u-\sqrt{\e_n}y)\phi(y)\dd y \dd u \dd s\\
&\to  \int_0^t \int_0^s \int_{\R^2}\nabla K^{\theta,\lambda}_{s-u}(\tX^{k}_s-\tX^{i}_u)\phi(y)\dd y \dd u \dd s =Z_t^{k,i}
\end{align*}
in $L^1(\Omega)$, because (recall that $x\to \nabla K^{\theta,\lambda}_{r} (x)$ is smooth for $r>0$ fixed)
\vip
\noindent $\bullet$ the family $(\nabla K^{\theta,\lambda}_{s-u}(\tX^{k,n}_s-\tX^{i,n}_u-\sqrt{\e_n}y), n\geq 1)$
is uniformly integrable with respect to the measure 
$\E[\int_0^t \int_0^s\int_{\R^2} \;\cdot\; \phi(y)\dd y\dd u \dd s]$ by \eqref{eq:Prop21-5} and since $\frac{2\gamma}{3}>1$;
\vip
\noindent $\bullet$ for all $s>u>0$, all $y\in \R^2$,  $\nabla K^{\theta,\lambda}_{s-u}(\tX^{k,n}_s-\tX^{i,n}_u - \sqrt{\e_n}y)$
a.s. tends to $\nabla K^{\theta,\lambda}_{s-u}(\tX^{k}_s-\tX^{i}_u)$.
\end{proof}

\section{From the two-species system to the nonlinear model}
\label{tstks}

Following the proofs of~\cite[Theorem~6]{fournier-jourdain} and~\cite[Theorem 14]{FT1}, we
finally give the

\begin{proof}[Proof of Theorem~\ref{main}-(iii)]
We use the shortened notation $C=C(\R_+,\R^2)$. 
We fix $\theta>0$, $\lambda>0$ and $\chi \in (0,\chi^*_\theta)$. By Remark~\ref{chistar},
there are $\alpha>0$ and $\gamma \in (\frac32,2)$ such that~\eqref{condistar} holds true. We consider
some sequences $(\e_N)_{N\geq 2}$ and $(M_N)_{N\geq 2}$ satisfying~\eqref{condie3}.
For each $N\geq 2$, we consider the $(\rho_0,c_0,N,M_N,\e_{N})$-2SKS system $(X^{k,N}_t)_{t\geq 0, k\in I_N}$
and  $\mu^N=\frac1N\sum_{k\in I_N}\delta_{(X^{k,N}_t)_{t\geq 0}}$.

\vip
{\it Step 1.} By~\eqref{condie3}, for all $N$ large enough, $(N,M_N,\e_N)\in \Aa_\sigma$. Thus by Point i),
the family $((X^{1,N}_t)_{t\geq 0} : N\geq 2)$ is tight in $C(\R_+,\R^2)$. Together with
exchangeability, this classically implies the tightness of $(\mu^{N}, N\geq 2)$ in $\Pp(C)$,
see Sznitman~\cite[Proposition~2.2]{Sznitman}.
We now consider a (not relabeled) subsequence of $\mu^N$ converging in law to some (possibly random)
$\mu\in \Pp(C)$ as $N\to \infty$. We have to show that $\mu$ a.s. solves $\mathcal{(MP)}(\rho_0,c_0)$.
\vip

We denote by $(\mu_t)_{t\geq 0}$ the family of time-marginals of $\mu$.
Since $\mu^N_0=\frac1N\sum_{k\in I_{N_n}}\delta_{X^i_0}$ goes to $\rho_0$ because the family $(X^i_0)_{i \geq 1}$
is i.i.d. with common law $\rho_0$, we have $\mu_0=\rho_0$ a.s.

\vip

Since $\mu^N$ converges
in law to $\mu$, it also holds true that $\mu^N \otimes \mu^N$
and $\mu^N \odot \mu^N$ both converge in law to $\mu\otimes \mu$
in $\Pp(C\times C)$, where
$$
\mu^N \odot \mu^N=
\frac 1 {N(N-1)}\sum_{i,j \in I_{N},i\neq j} \delta_{((X^{i,N}_t)_{t\geq 0},(X^{j,N}_t)_{t\geq 0})}. 
$$

{\it Step 2.} Here we show that $\mu$ a.s. satisfies \eqref{eq:ConditionMP}.
By the Fatou lemma, for all $t\geq 0$,
\begin{align*}
\E\Big[\int_0^t& \int_0^s \int_{\R^2} \int_{\R^2}  
\frac {\mu_u(\dd y)\mu_s(\dd x) \dd u \dd s}{({s-u}+|x-y|^2)^\gamma}\Big]
= \E\Big[\int_C\int_C\int_0^t \int_0^s  \frac { \dd u \dd s}{({s-u}+|x_s-y_u|^2)^\gamma}(\mu\otimes \mu)
(\dd x,\dd y)\Big]\\
=&\E\Big[\int_C\int_C\int_0^t \int_0^s \int_{\R^2} \frac {\phi(z)\dd z \dd u \dd s}{({s-u}+|x_s-y_u|^2)^\gamma}(\mu\otimes \mu)
(\dd x,\dd y)\Big] \\
\leq & \liminf_N \E\Big[\int_C\int_C\int_0^t \int_0^s \int_{\R^2}\frac { \dd u  \dd s}{({s-u}+|x_s-y_u-\sqrt{\e_N}z|^2)^\gamma} \phi(z)\dd z
(\mu^N \odot \mu^N) (\dd x,\dd y) \Big]\\
=& \liminf_N \frac 1{N(N-1)}\sum_{i,j \in I_N,i\neq j}\E\Big[\int_0^t \int_0^s \int_{\R^2}\frac {\phi(z)\dd z\dd u  \dd s}
{({s-u}+|X^{i,N}_s-X^{j,N}_u-\sqrt{\e_N}z|^2)^\gamma}\Big]<\infty
\end{align*}
by~\eqref{eq:Prop21-2} and exchangeability. By~\eqref{tri0}-\eqref{tri1p}, $|K^{\theta,\lambda}_s(x)| \leq \frac{A}{s+|x|^2}$ and $|\nabla K^{\theta,\lambda}_s(x)| \leq \frac{A}{(s+|x|^2)^{3/2}}$. Hence
\begin{equation}\label{tbrr}
\E\Big[\int_0^t \int_0^s \int_{\R^2} \int_{\R^2} \Big([K^{\theta,\lambda}_{s-u}(x-y)]^{\gamma} +|\nabla K^{\theta,\lambda}_{s-u}(x-y)|^{\frac{2\gamma}3}\Big) \mu_u(\dd y)\mu_s(\dd x) \dd u \dd s  \Big]<\infty.
\end{equation}
Since $\gamma>\frac 32$, \eqref{eq:ConditionMP} follows.

\vip

{\it Step 3.} We finally check that a.s., for any $\varphi\in C^2_c(\R^2)$, the process $(M_t^\varphi)_{t\geq 0}$ 
defined in~\eqref{def_mart} (with $\Q_r$ replaced by $\mu_r$) is a $\mu$-martingale. It suffices that for all $t>s>0$, 
all continuous bounded function $\Phi:C\to \R$, we have $\Theta(\mu\otimes\mu)=0$ a.s., where
for $\Pi \in \Pp(C\times C)$,
\begin{align*}
\Theta(\Pi)=\int_{C\times C}\!\!\!\!\!
\Phi((x_r)_{r\in [0,s]})
\Big(\varphi(x_t)-\varphi(x_s)-&
\int_s^t \Big[ \Delta \varphi(x_u) + \chi \nabla \varphi(x_u)\cdot \nabla b^{c_0,\theta,\lambda}_s(x_u)\\
&+\chi \nabla \varphi(x_u)\cdot  \int_0^u \nabla K^{\theta,\lambda}_{u-v}(x_u-y_v) \dd v \Big]\dd u \Big) 
\Pi(\dd x,\dd y).
\end{align*}
To this end, we will prove that $\E[|\Theta(\mu\otimes\mu)|]=0$.
\vip

We introduce $\Theta_\e$ defined as $\Theta$ replacing $\nabla K^{\theta,\lambda}_{u-v}$
by $\nabla(\phi_\e \ast K^{\theta,\lambda}_{u-v})$ and first verify that
\begin{equation}\label{scc1}
\lim_N \E[|\Theta_{\e_N}(\mu^N\odot \mu^N)|]=0.
\end{equation}
We have
\begin{align*}
&\Theta_{\e_N}(\mu^N\odot \mu^N)\\
=&\frac1{N}\sum_{i\in I_{N}}
\Phi((X^{i,N}_r)_{r\in [0,s]}) 
\Big(\varphi(X^{i,N}_t)-\varphi(X^{i,N}_s) -\int_s^t \Big[ \Delta \varphi(X^{i,N}_u)+\chi \nabla \varphi(X^{i,N}_u)\cdot \nabla b^{c_0,\theta,\lambda}_s(X^{i,N}_u)\\
& \hskip4cm
+ \frac{\chi}{N-1}\sum_{j\in I_{N},j\neq i}\nabla \varphi(X^{i,N}_u)\cdot 
\int_0^u \nabla (\phi_{\e_N}\ast K^{\theta,\lambda}_{u-v})(X^{i,N}_u-X^{j,N}_v) \dd v \Big]\dd u \Big)\\
=&\frac 1{N}\sum_{i\in I_{N}}
\Phi((X^{i,N}_r)_{r\in [0,s]})(O^{i,N}_t-O^{i,N}_s+V^{i,N}_t-V^{i,N}_s),
\end{align*}
where, with $(\nu^{i,N}_t)_{t\geq 0,i \in I_N}$ the empirical chemoattractant concentrations corresponding to
2SKS system $(X^{i,N}_t)_{t\geq 0, i\in I_N}$,
\begin{align*}
O^{i,N}_t=& \varphi(X^{i,N}_t) - \int_0^t \Delta \varphi(X^{i,N}_u)\dd u
-\chi \int_0^t \nabla\varphi(X^{i,N}_u) \cdot (\nabla \phi_{\e_N}\ast \nu^{i,N}_u)(X^{i,N}_u) \dd u,\\
V^{i,N}_t=& \chi \int_0^t \nabla\varphi(X^{i,N}_u) \cdot Z^{i,N}_u \dd u,\\
Z^{i,N}_u=&  (\nabla \phi_{\e_n}\ast \nu^{i,N}_u)(X^{i,N}_u)-\nabla b^{c_0,\theta,\lambda}_u(X^{i,N}_u)
- \frac1{N-1}\sum_{j\in I_{N},j\neq i}\int_0^u \nabla(\phi_{\e_N}\ast K^{\theta,\lambda}_{u-v})(X^{i,N}_u-X^{j,N}_v) \dd v.
\end{align*}
Thus by exchangeability, for some constant $A$ (depending on $\Phi$ and $\varphi$),
\begin{equation}\label{ar1}
\E[|\Theta_{\e_N}(\mu^N\odot \mu^N)|]\leq \E\Big[ \Big(\frac 1{N}\sum_{i\in I_{N}}
\Phi((X^{i,N}_r)_{r\in [0,s]})(O^{i,N}_t-O^{i,N}_s)\Big)^2 \Big]^{\frac12} + A\int_0^t\E[|Z^{1,N}_u|]\dd u.
\end{equation}
The Itô formula, starting from the first line of~\eqref{troc}, tells us that
\begin{align*}
O^{i,N}_t= \varphi(X^{i}_0) + \sqrt 2 \int_0^t \nabla \varphi(X^{i,N}_u) \cdot \dd B^i_u.
\end{align*}
Since $\Phi$ and $\nabla \varphi$ are bounded and since $B^1,\dots,B^N$ are independent, 
one easily checks that
\begin{equation}\label{ar2}
\E\Big[ \Big(\frac 1{N}\sum_{i\in I_{N}}
\Phi((X^{i,N}_r)_{r\in [0,s]})(O^{i,N}_t-O^{i,N}_s)\Big)^2 \Big] \leq \frac{A}N.
\end{equation}
Next, we write $Z^{1,N}_u=U^{1,N}_u+R^{1,N}_u$, where
\begin{align*}
U^{1,N}_u=&(\nabla \phi_{\e_N}\ast \nu^{1,N}_u)(X^{1,N}_u)\!-\! \nabla (\phi_{\e_N}\ast b^{c_0,\theta,\lambda}_u)(X^{1,N}_u)\\
&\!-\!\frac1{N-1}\sum_{j\in I_{N},j\neq 1}\!\int_0^u\! \nabla (\phi_{\e_N}\ast K^{\theta,\lambda}_{u-v})(X^{1,N}_u-X^{j,N}_v) \dd v,\\
R^{1,N}_u=&\nabla (\phi_{\e_N}\ast b^{c_0,\theta,\lambda}_u)(X^{1,N}_u)-\nabla b^{c_0,\theta,\lambda}_u(X^{1,N}_u).
\end{align*}
Applying Proposition~\ref{prel}, we get
\begin{align}\label{ar3}
\int_0^t\E[|U^{1,N}_u|]\dd u \!\leq\! t \Big(\!\sup_{u\in [0,t]}\E[|U^{1,N}_u|^2]\Big)^{\frac12}
\!\leq\! A_t \Big(\frac{1}{(NM_N\e_N^{1/2})^2}\!+\!\frac{1}{NM_N\e_N^{3+\sigma}}\!+\!\frac{1}{NM_N^{1+\sigma}\e_N^{2+\sigma}}
\Big)^{\frac12},
\end{align}
which tends to $0$ as $N\to \infty$ by~\eqref{condie3} (and since $NM_N\e_N^{1/2}\geq NM_N\e_N^{3+\sigma}$).
Finally, $\E[\int_0^t |R^{1,N}_u| \dd u] \to 0$ as $N\to \infty$ by~\eqref{rtt2}. This, together with~\eqref{ar1}, \eqref{ar2} and \eqref{ar3} implies~\eqref{scc1}.

\vip

Next we introduce, for $\eta\in (0,1]$, 
$\Theta'_\eta$ defined as $\Theta$ with $\nabla K^{\theta,\lambda}_s(x)$ replaced by the continuous and 
bounded kernel $\nabla K^{\theta,\lambda,\eta}_s(x)=
\frac{s^2}{(s+\eta)^2}\nabla K^{\theta,\lambda}_s(x)$, recall~\eqref{gkb}. 
As in~\cite[proof of Theorem~14, Step~2.2]{FT1},
$\Pi \mapsto \Theta'_\eta(\Pi)$ is continuous and bounded from $\Pp(C\times C)$ to $\R$. 
Since $\mu^N\odot\mu^N$ goes in law to $\mu \otimes \mu$,
\begin{equation}\label{scc2}
\E[|\Theta'_\eta(\mu\otimes \mu)|]=\lim_{N} \E[|\Theta'_\eta(\mu^N\odot\mu^N)|].
\end{equation}

As in~\cite[proof of Theorem~14, Step~2.3]{FT1}, \eqref{tbrr} implies that
\begin{equation}\label{pld}
\lim_{\eta\to 0} \E[|\Theta(\mu\otimes \mu)-\Theta'_\eta(\mu\otimes \mu)|]=0.
\end{equation}

We finally verify that 
\begin{equation}\label{pld2}
\lim_{\eta\to 0} \limsup_{N} \E[|\Theta_{\e_N}(\mu^N\odot \mu^N)-\Theta'_\eta(\mu^N\odot \mu^N)|]=0.
\end{equation}
We write, using exchangeability,
\begin{align*}
\Delta_{\eta,N}=& \E[|\Theta_{\e_N}(\mu^N\odot \mu^N)-\Theta'_\eta(\mu^N\odot \mu^N)|]\\
\leq & C \E\Big[\int_0^t \int_0^u |\nabla(\phi_{\e_N}\ast K^{\theta,\lambda}_{u-v})(X^{1,N}_u-X^{2,N}_v) - \nabla K^{\theta,\lambda,\eta}_{u-v}(X^{1,N}_u-X^{2,N}_v)| \dd v\dd u  \Big]\\
\leq & C\Delta^1_{\eta,N}+C\Delta^2_{\eta,N},
\end{align*}
where, using that $\nabla(\phi_{\e_N}\ast K^{\theta,\lambda,\eta}_{u-v}(x))= \frac{(u-v)^2}{(u-v+\eta)^2}\nabla(\phi_{\e_N}\ast K^{\theta,\lambda}_{u-v}(x))$,
\begin{align*}
\Delta^1_{\eta,N}=&\E\Big[\int_0^t \int_0^u \Big(1-\frac{(u-v)^2}{(u-v+\eta)^2}\Big)|\nabla(\phi_{\e_N}\ast K^{\theta,\lambda}_{u-v})(X^{1,N}_u-X^{2,N}_v)|\dd v \dd u  \Big], \\
\Delta^2_{\eta,N}=&\E\Big[\int_0^t \int_0^u 
|\nabla(\phi_{\e_N}\ast K^{\theta,\lambda,\eta}_{u-v})(X^{1,N}_u-X^{2,N}_v) - \nabla K^{\theta,\lambda,\eta}_{u-v}(X^{1,N}_u-X^{2,N}_v)| \dd v \dd u\Big].
\end{align*}
Fix $\eta \in (0,1]$. Since $\nabla K^{\theta,\lambda,\eta}_{u-v}$ and $\nabla (\phi_{\e_N}\ast K^{\theta,\lambda,\eta}_{u-v})$ are bounded (uniformly in $v>u>0$ and $N\geq 2$)
and since $\lim_N||\nabla (\phi_{\e_N}\ast K^{\theta,\lambda,\eta}_{u-v})- \nabla K^{\theta,\lambda,\eta}_{u-v} ||_\infty =0$ for each $v>u>0$ (because $\nabla K^{\theta,\lambda,\eta}_{u-v}$
is uniformly continuous on $\R^2$), we get $\limsup_N \Delta^2_{\eta,N}=0$ by dominated convergence. Next,
\begin{align*}
\Delta^1_{\eta,N} \leq & \E\Big[\int_0^t \int_0^u \int_{\R^2}\Big(1-\frac{(u-v)^2}{(u-v+\eta)^2}\Big)|\nabla K^{\theta,\lambda}_{u-v}(X^{1,N}_u-X^{2,N}_v-\sqrt{\e_N}y)|\phi(y)\dd y\dd v \dd u  \Big]\\
\leq &  \rho_\eta \E\Big[\int_0^t \int_0^u \int_{\R^2} |\nabla K^{\theta,\lambda}_{u-v}(X^{1,N}_u-X^{2,N}_v-\sqrt{\e_N}y)|^{\frac{2\gamma}3}\phi(y)\dd y\dd v \dd u  \Big]^{\frac3{2\gamma}}
\end{align*}
by Hölder's inequality, where $\rho_\eta=[\int_0^t \int_0^u (1-\frac{(u-v)^2}{(u-v+\eta)^2})^{\frac{2\gamma}{2\gamma-3}}
\dd v \dd u]^{\frac{2\gamma-3}{2\gamma}}$.
Using~\eqref{eq:Prop21-5}, we conclude that $\sup_N \Delta^1_{\eta,N} \leq C \rho_\eta$, which tends to $0$ as $\eta\to 0$ by dominated convergence. We have shown~\eqref{pld2}.

\vip
To complete the proof, we notice that for all $N\geq2$, all $\eta \in(0,1]$,
\begin{align*}
\E[|\Theta(\mu\otimes\mu)|] \leq& \E[|\Theta(\mu\otimes\mu) -\Theta'_\eta(\mu\otimes\mu)|]\\  
&+\big|\E[|\Theta'_\eta(\mu\otimes\mu)|] - \E[|\Theta'_\eta(\mu^N\odot\mu^N)]\big|\\
&+\E[|\Theta'_\eta(\mu^N\odot\mu^N)-\Theta_{\e_N}(\mu^N\odot\mu^N)|]\\
&+\E[|\Theta_{\e_N}(\mu^N\odot\mu^N)|].
\end{align*}
Taking the limsup as $N\to \infty$ and then the limit as $\eta\to 0$, we conclude that $\E[|\Theta(\mu\otimes \mu)|]=0$ by~\eqref{scc1}-\eqref{scc2}-\eqref{pld}-\eqref{pld2}.
\end{proof}

\appendix

\section{Technical lemmas}\label{apa}

We first study the function $\varphi^\e_{t,x}$ used in Proposition~\ref{prel}.
Recall that $g_t$ and  $K^{\theta,\lambda}_t$ were defined in~\eqref{gkb}, while $\phi_\e$ was 
introduced in Assumption~\ref{mol}.

\begin{lemma}\label{tri} 
For $\e>0$, $x\in \R^2$ and $t>0$ fixed and for
$s\in [0,t]$ and $y\in \R^2$, set 
$$
\varphi^\e_{t,x}(s,y)=\theta(\nabla \phi_\e \ast K^{\theta,\lambda}_{t-s})(x-y)\in\R^2.
$$
It holds that (these are equalities in $\R^2$)
\begin{equation}\label{tri2}
\left\{\begin{array}{ll}
\partial_s \varphi^\e_{t,x}(s,y)+\frac1\theta(\Delta_y \varphi^\e_{t,x}(s,y) -\lambda 
\varphi^\e_{t,x}(s,y))=0, & s\in [0,t],y \in \R^2,\\[3pt]
\varphi^\e_{t,x}(t,y)=\nabla \phi_\e(x-y),&y \in \R^2.
\end{array}\right.
\end{equation}
For all $r\geq 2$, there exists $A_{r}$ (also depending on $\theta$)
such that for all $t\geq s \geq 0$, all $y\in \R^2$,
\begin{gather}
\int_{\R^2} |\varphi^\e_{t,x}(s,y)|^r \dd x \leq \frac{A_r}{(\e+t-s)^{\frac{3r}2-1}}
\quad \text{and} \quad 
\int_{\R^2} |\nabla_x\varphi^\e_{t,x}(s,y)|^r \dd x \leq \frac{A_r}{(\e+t-s)^{2r-1}},\label{estp0}\\
\int_0^t \int_{\R^2} |\varphi^\e_{t,x}(s,y)|^r \dd x \dd s \leq \frac{A_r}{\e^{\frac{3r}2-2}}
\quad \text{and} \quad 
\int_0^t \int_{\R^2} |\nabla_x\varphi^\e_{t,x}(s,y)|^r \dd x \dd s \leq \frac{A_r}{\e^{2r-2}},\label{estp1}\\
\int_0^t \int_{\R^2} |\nabla_y\varphi^\e_{t,x}(s,y)|^r \dd x \dd s \leq \frac{A_r}{\e^{2r-2}}
\quad \text{and} \quad 
\int_0^t \int_{\R^2} |\nabla_x\nabla_y\varphi^\e_{t,x}(s,y)|^r \dd x \dd s \leq \frac{A_r}{\e^{\frac{5r}2-2}}.
\label{estp2}
\end{gather}
\end{lemma}

\begin{proof}
First, \eqref{tri2} follows from the facts that 
$\partial_t K^{\theta,\lambda}_t(x)=\frac1\theta(\Delta K^{\theta,\lambda}_t(x)-\lambda K^{\theta,\lambda}_t(x))$ and that
$\theta K^{\theta,\lambda}_0=g_0=\delta_0$.
Next, we have $\varphi^\e_{t,x}(s,y)=e^{-\frac \lambda\theta t} \nabla(\phi_\e\ast g_{2(t-s)/\theta})(x-y)$, so that
\begin{align*}
\int_{\R^2} |\varphi^\e_{t,x}(s,y)|^r \dd x \leq& ||\nabla(\phi_\e\ast g_{2(t-s)/\theta})||_{L^r(\R^2)}^r \\
\leq& \min\Big\{||\phi_\e||_{L^1(\R^2)}^r||\nabla g_{2(t-s)/\theta}||_{L^r(\R^r)}^r,
||\nabla \phi_\e||_{L^r(\R^2)}^r||g_{2(t-s)/\theta}||_{L^1(\R^r)}^r \Big\}
\end{align*}
by Young's convolution inequality. Easy computations show that $||\phi_\e||_{L^1(\R^2)}=
||g_{2(t-s)/\theta}||_{L^1(\R^r)}=1$, that $||\nabla g_{2(t-s)/\theta}||_{L^r(\R^r)}^r\leq A_r (t-s)^{1-\frac{3r}2}$
and that $||\nabla \phi_\e||_{L^r(\R^2)}^r \leq A_r \e^{1-\frac{3r}2}$. Hence
\begin{align*}
\int_{\R^2} |\varphi^\e_{t,x}(s,y)|^r \dd x \leq A_r \min\{(t-s)^{1-\frac{3r}2},\e^{1-\frac{3r}2}\}\leq 
A_r(\e+t-s)^{1-\frac{3r}2}.
\end{align*}
This shows the first inequality in~\eqref{estp0} which, integrated in $s\in [0,t]$, implies the first inequality 
in~\eqref{estp1}. Similarly,
\begin{align*}
\int_{\R^2} |\nabla_x\varphi^\e_{t,x}(s,y)|^r \dd x \leq& ||D^2(\phi_\e\ast g_{2(t-s)/\theta})||_{L^r(\R^2)}^r \\
\leq& \min\Big\{||\phi_\e||_{L^1(\R^2)}^r||D^2 g_{2(t-s)/\theta}||_{L^r(\R^r)}^r,
||D^2\phi_\e||_{L^r(\R^2)}^r||g_{2(t-s)/\theta}||_{L^1(\R^r)}^r \Big\}.
\end{align*}
But $||D^2 g_{2(t-s)/\theta}||_{L^r(\R^r)}^r\leq A_r (t-s)^{1-2r}$,
and $||D^2\phi_\e||_{L^r(\R^2)}^r \leq A_r \e^{1-2r}$, whence
\begin{align*}
\int_{\R^2} |\nabla_x\varphi^\e_{t,x}(s,y)|^r \dd x \leq A_r \min\{(t-s)^{1-2r},\e^{1-2r}\}\leq 
A_r(\e+t-s)^{1-2r}.
\end{align*}
This shows the second inequality in~\eqref{estp0} which, integrated in $s\in [0,t]$, implies the second inequality 
in~\eqref{estp1}. The first inequality in~\eqref{estp2} is shown similarly. Finally,
\begin{align*}
\int_{\R^2} |\nabla_x \nabla_y\varphi^\e_{t,x}(s,y)|^r \dd x \leq& ||D^3(\phi_\e\ast g_{2(t-s)/\theta})||_{L^r(\R^2)}^r \\
\leq& \min\Big\{||\phi_\e||_{L^1(\R^2)}^r||D^3 g_{2(t-s)/\theta}||_{L^r(\R^r)}^r,
||D^3\phi_\e||_{L^r(\R^2)}^r||g_{2(t-s)/\theta}||_{L^1(\R^r)}^r \Big\}.
\end{align*}
But $||D^3 g_{2(t-s)/\theta}||_{L^r(\R^r)}^r\leq A_r (t-s)^{1-\frac{5r}2}$,
and $||D^3\phi_\e||_{L^r(\R^2)}^r \leq A_r \e^{1-\frac{5r}2}$, whence
\begin{align*}
\int_{\R^2} |\nabla_x\varphi^\e_{t,x}(s,y)|^r \dd x \leq A_r \min\{(t-s)^{1-\frac{5r}2},\e^{1-\frac{5r}2}\}\leq 
A_r(\e+t-s)^{1-\frac{5r}2}
\end{align*}
which, integrated in $s\in [0,t]$, yields the last inequality in~\eqref{estp2}.
\end{proof}

We next establish some more estimates and properties related to $K^{\theta,\lambda}_t$ and $b^{c_0,\theta,\lambda}_t$, that were defined in~\eqref{gkb} and to $\phi_\e$, which 
was introduced in Assumption~\ref{mol}.

\begin{lemma} We fix $c_0 \in L^\infty(\R^2)\cap L^1(\R^2)$ and $\theta>0$, $\lambda>0$.
There exists a constant $A$ such that for all $t> 0$, all $x\in\R^2$, all $\e\in [0,1]$,
\begin{equation}\label{tri0}
|K^{\theta,\lambda}_t(x)|\leq A (t+|x|^2)^{-1} \quad \text{and} \quad 
|(\phi_\e \ast K^{\theta,\lambda}_t)(x)|\leq A \int_{\R^2} (t+|x-\sqrt \e y|^2)^{-1} \phi(y)\dd y.
\end{equation}
For any $\alpha>0$, any $t>0$, any $x\in \R^2$, any $\e\in [0,1]$,
\begin{align}
|\nabla K^{\theta,\lambda}_t(x)|\leq &
C_3(\theta,\alpha)(t+\alpha |x|^2)^{-\frac{3}{2}}, \quad \text{where} \quad C_3(\theta,\alpha)=\frac{\sqrt \theta C_0\big(\frac{4\alpha}{\theta}\big)}{4\pi}, \label{tri1p}\\
|\nabla (\phi_\e \ast K^{\theta,\lambda}_t)(x)|\leq& 
C_3(\theta,\alpha)\int_{\R^2}(t+\alpha |x-\sqrt\e y|^2)^{-\frac{3}{2}}\phi(y)\dd y.
\label{tri1}
\end{align}
There exists $A>0$ depending on $\theta$, $c_0$ and $\phi$ such that for any $t>0$, any $\e\in [0,1]$,
\begin{equation}\label{rmu}
||\nabla (\phi_{\e}\ast b^{c_0,\theta,\lambda}_t) ||_{L^\infty(\R^2)}\leq  A (\e+t)^{-\frac 1 2}.
\end{equation}
For all $t\geq 0$, it holds that
\begin{equation}\label{rtt2}
\lim_{\e\to 0}\int_0^t ||\nabla \phi_{\e}\ast b^{c_0,\theta,\lambda}_{s} - \nabla b^{c_0,\theta,\lambda}_{s}||_{L^\infty(\R^2)} 
\dd s=0.
\end{equation}
For all $s>0$, if $x_n\to x$ and $\e_n\to 0$, then
\begin{equation}\label{rtt1}
(\nabla \phi_{\e_n}\ast b^{c_0,\theta,\lambda}_{s})(x_n)\to \nabla b^{c_0,\theta,\lambda}_{s}(x).
\end{equation}
\end{lemma}

\begin{proof} 
The first inequality in~\eqref{tri0} has been checked in~\cite[proof of Theorem~14, Step~1]{FT1}, and the second inequality in~\eqref{tri0} follows, recalling that  
$\phi_\e(y)=\e^{-1}\phi(\e^{-\frac12}y)$
and proceeding to the substitution $u=\e^{-\frac12}y$.
Next, \eqref{tri1p} has been checked in~\cite[Remark~9]{FT1},
and~\eqref{tri1} follows, using that 
$|\nabla (\phi_\e \ast K^{\theta,\lambda}_t)(x)|\leq (\phi_\e\ast |\nabla K^{\theta,\lambda}_t|)(x)$,  recalling that $\phi_\e(y)=\e^{-1}\phi(\e^{-\frac12}y)$
and proceeding to the substitution $u=\e^{-\frac12}y$.
\vip
We turn to~\eqref{rmu}. Recalling~\eqref{gkb}, we have
$\nabla (\phi_{\e}\ast b^{c_0,\theta,\lambda}_t) = e^{-\frac\lambda\theta t}
\nabla (\phi_\e \ast g_{2t/\theta}\ast c_0)$, whence
$$
||\nabla (\phi_{\e}\ast b^{c_0,\theta,\lambda}_t)||_{L^\infty(\R^2)}\leq ||\nabla (\phi_\e \ast g_{2t/\theta})||_{L^1(\R^2)}
||c_0||_{L^\infty(\R^2)} \leq A (t+\e)^{-\frac 12},
$$
because by Young's inequality,
\begin{align*}
||\nabla (\phi_\e \ast g_{2t/\theta})||_{L^1(\R^2)}\leq &
\min\Big\{||\nabla \phi_\e||_{L^1(\R^2)} ||g_{2t/\theta}||_{L^1(\R^2)},||\phi_\e||_{L^1(\R^2)} 
||\nabla g_{2t/\theta}||_{L^1(\R^2)}\Big\} \\
\leq& A \min\{\e^{-\frac12},t^{-\frac12}\},
\end{align*}
which is smaller than $A(\e+t)^{-\frac 12}$. This shows~\eqref{rmu}. 

\vip
We next verify~\eqref{rtt2}. Since 
$||\nabla \phi_{\e}\ast b^{c_0,\theta,\lambda}_s-\nabla b^{c_0,\theta,\lambda}_s ||_{L^\infty(\R^2)}\leq A s^{-1/2}$ 
by~\eqref{rmu}, it suffices, by dominated convergence, to show that for each $s>0$,
$\lim_{\e\to 0}||\nabla \phi_{\e}\ast b^{c_0,\theta,\lambda}_s-\nabla b^{c_0,\theta,\lambda}_s ||_{L^\infty(\R^2)}=0$.
But as previously,
$$
||\nabla \phi_{\e}\ast b^{c_0,\theta,\lambda}_s-\nabla b^{c_0,\theta,\lambda}_s ||_{L^\infty(\R^2)}
\leq ||c_0||_{L^\infty(\R^2)}  ||\phi_\e\ast \nabla g_{2s/\theta}-\nabla g_{2s/\theta}||_{L^1(\R^2)},
$$
which classically tends to $0$ as $\e\to 0$ because$\nabla g_{2s/\theta}\in L^1(\R^2)$.
\vip

If $s>0$, $x_n\to x$ and $\e_n\to 0$, then
$$
|(\nabla \phi_{\e_n}\ast b^{c_0,\theta,\lambda}_{s})(x_n)- \nabla b^{c_0,\theta,\lambda}_{s}(x)|\leq
||\nabla \phi_{\e_n}\ast b^{c_0,\theta,\lambda}_{s}- \nabla b^{c_0,\theta,\lambda}_{s}||_\infty+
|\nabla b^{c_0,\theta,\lambda}_{s}(x_n)-\nabla b^{c_0,\theta,\lambda}_{s}(x)|.
$$
We have seen 5 lines above that the first term on the RHS tends to $0$. The second one also tends to $0$,
since $y\mapsto \nabla b^{c_0,\theta,\lambda}_s (x)= e^{-\lambda t} (\nabla g_{2t/\theta}\ast c_0)(y)$ is continuous. 
\end{proof}

\small
\bibliography{biblio}

\end{document}